\documentclass[a4paper,11pt]{article}
\usepackage{amsmath,amsthm,amssymb,amscd,epsfig}

\newtheorem{thm}{Theorem}[section]

\newtheorem{cor}[thm]{Corollary}
\newtheorem{lemma}[thm]{Lemma}
\newtheorem{remark}[thm]{Remark}

\newtheorem{theorem}{Theorem}

\newtheorem{quest}[thm]{Question}

\newtheorem{conjecture}[thm]{Conjecture}

\numberwithin{equation}{section}

\def\R{\mathbb R}
\def\C{\mathbb C}
\def\N{\mathbb N}
\def\D{\mathbb D}

\def\H{\mathcal H}

\def\G{\mathcal G}

\def\diam{\text{diam}}

\def\ds{\displaystyle}

\title{Sharp Examples for Planar Quasiconformal Distortion of Hausdorff Measures and Removability}

\author{
\textit{I. Uriarte-Tuero}
\thanks{{
Uriarte-Tuero is a postdoctoral fellow in the Department of Mathematics of the University of Missouri-Columbia}\newline
\newline Keywords   Quasiconformal, Hausdorff measure, Removability}}

\begin{document}

\maketitle

\begin{abstract}


In the celebrated paper \cite{astalaareadistortion}, Astala showed optimal area distortion bounds and dimension distortion estimates for planar quasiconformal mappings. He asked (Question 4.4) whether a finer result held, namely absolute continuity of Hausdorff measures under push-forward by quasiconformal mappings. This was proven in one particular case relevant for removability questions, in joint work of Astala, Clop, Mateu, Orobitg and the author \cite{astalaclopmateuorobitguriartepreprint} (Theorem 1.1), the other cases remaining open. A related question that we left open in \cite{astalaclopmateuorobitguriartepreprint} (Question 4.2) (which was asked by Astala to the author before \cite{astalaclopmateuorobitguriartepreprint} in an equivalent form \cite{astalapersonalcommunication}) is whether BMO removability for $K$-quasiregular mappings and ($L^{\infty}$) removability for $K$-quasiregular mappings are indeed different problems.

In this paper we give a series of examples answering in the positive Question 4.2 in \cite{astalaclopmateuorobitguriartepreprint}, at the same time proving sharpness in two different senses of Theorem 1.1 in \cite{astalaclopmateuorobitguriartepreprint}, and also giving examples that would yield sharpness in those two different senses as well for the absolute continuity of Hausdorff measures under push-forward by quasiconformal mappings, were it to be proven.
\end{abstract}

\section{Introduction}

An orientation preserving homeomorphism $\phi:\Omega\rightarrow\Omega'$ between planar domains $\Omega,\Omega' \subset \C$ is called {\it{$K$-quasiconformal}} if  it belongs to  the Sobolev space $W^{1,2}_{loc}(\Omega)$ and satisfies the {\it{distortion inequality}}
\begin{equation}\label{distortioninequality}
\max_\alpha|\partial_\alpha\phi|\leq K\min_\alpha|\partial_\alpha\phi| \,\,\,\,\,\text{a.e. in }\Omega
\end{equation}
Quasiconformal mappings preserve sets of zero Lebesgue measure (See the work of Ahlfors \cite{ahlfors}.) They also preserve sets of zero Hausdorff dimension, since $K$-quasiconformal mappings are H\"older continuous with exponent $1/K$, see \cite{mori}. However, these maps do not preserve Hausdorff dimension in general, and  it was in the celebrated paper \cite{astalaareadistortion} where the precise dimension distortion bounds were given. Namely, for any compact set $E$ with dimension $t$ and for any $K$-quasiconformal mapping $\phi$ we have
\begin{equation}\label{distortionofdimension}
\frac{1}{K}\left(\frac{1}{t}-\frac{1}{2}\right)\leq\frac{1}{\dim(\phi(E))}-\frac{1}{2}\leq K\left(\frac{1}{t}-\frac{1}{2}\right)
\end{equation}
These bounds are optimal, i.e. equality may occur in either estimate.

A finer question fundamental to the understanding of size distortion by quasiconformal mappings was raised in \cite{astalaareadistortion} (Question 4.4.): whether the estimates (\ref{distortionofdimension}) can be improved to the level of Hausdorff measures $\H^t$. In other words, if $\phi$ is a planar $K$-quasiconformal mapping, $0<t<2$ and $t'=\frac{2Kt}{2+(K-1)t}$, the question is whether it is  true that
\begin{equation}\label{abscont}
\H^t(E)=0\,\,\,\Longrightarrow\,\,\,\H^{t'}(\phi(E))=0,
\end{equation}
or equivalently, $\phi^\ast\H^{t'}\ll\H^t$. The above classical results of Ahlfors and Mori assert that this is true when $t=0$ or $t=2$. For the Lebesgue measure one has even precise quantitative bounds  $$|\phi(E)|\leq C\,|E|^\frac{1}{K}, $$
which also lead to the sharp Sobolev regularity,  $\phi\in W^{1,p}_{loc}(\C)$ for every  $p<\frac{2K}{K-1}$ (see \cite{astalaareadistortion}.)

Two important results towards \eqref{abscont} and related questions were given in \cite{astalaclopmateuorobitguriartepreprint} (which we have also used as a source for some parts of this paper.) Namely, 

\begin{theorem}\label{theorem1.1inACMOU}\textbf{[Theorem 1.1 in \cite{astalaclopmateuorobitguriartepreprint}]}
Let $\phi$ be a planar $K$-quasiconformal mapping, and let $E$ be a compact set. Then,
\begin{equation}\label{festimate}
\H^\frac{2}{K+1}(E)=0\,\,\,\Longrightarrow\,\,\,\H^{1}(\phi(E))=0
\end{equation}
\end{theorem}

which proves \eqref{abscont} for $t=\frac{2}{K+1}$, i.e. for image dimension  $t'=1$, and also the related 

\begin{theorem}\label{theorem2.5inACMOU}\textbf{[Theorem 2.5 in \cite{astalaclopmateuorobitguriartepreprint}]}
Let $E\subset\C$ be a compact set, and $\phi:\C\rightarrow\C$ a $K$-quasiconformal mapping. If $\, \H^\frac{2}{K+1}(E)$ is finite (or even $\sigma$-finite), then $\H^1(\phi(E))$ is $\sigma$-finite.
\end{theorem}

These two theorems have important applications to removability questions for quasiregular mappings. (See Theorem \ref{theorems1.2and4.3inACMOU} below.)

It is worth noting that a positive answer to Question 2.4 in \cite{astalaclopmateuorobitguriartepreprint} would prove \eqref{abscont}. However, this question has a negative answer, as was shown by a counterexample of Bishop \cite{bishopdistortionofdisksbyconformalmaps}.

Recall that an orientation preserving $f$ is a {\it{$K$-quasiregular mapping}} in a domain $\Omega\subset\C$ if  $f\in W^{1,2}_{loc}(\Omega)$ and $f$ satisfies the distortion inequality (\ref{distortioninequality}). When $K=1$, this class agrees with the class of analytic functions on $\Omega$. The classical {\it{Painlev\'e problem}} consists of giving metric and geometric characterizations of those sets $E$ that are removable for bounded analytic functions. Painlev\'e's theorem tells us that sets with $\H^1(E)=0$ (zero length) are removable, while Ahlfors \cite{ahlforsboundedanalyticfunctions} showed that no set of Hausdorff dimension  $> \;1$ has this property. In dimension $1$ the question is quite delicate. For the related $BMO$-problem (i.e. changing ``bounded" by BMO in the previous problem), Kaufman \cite{kaufman} and Kr\'{a}l \cite{kral} proved that the condition $\H^1(E)=0$ is a precise characterization for removable singularities of  $BMO$ analytic functions. Thus for analytic removability, dimension $1$ is the critical point both for  $L^\infty$ and  $BMO$. However, the solution to the original Painlev\'e problem lies much deeper and was only recently achieved by Tolsa (\cite{tolsasemiadditivityanalyticcapacity},\cite{tolsabilipschitzinvarianceanalyticcapacity}) in terms of curvatures of measures. Under the assumption that $\H^1(E)$ is finite, Painlev\'e's problem was earlier solved by G. David \cite{davidunrectifiable1setszeroanalyticcapacity}, who showed that a set $E$ with $0< \H^1(E)< \infty$ is removable for bounded analytic functions if it is purely unrectifiable. (The converse direction is due to Garabedian and Calder\'{o}n, see \cite{calderonlipschitzcurves}, \cite{garabedian}.) The countable semiadditivity of analytic capacity, due to Tolsa \cite{tolsasemiadditivityanalyticcapacity}, implies that this result remains true if we only assume $\H^1(E)$ to be $\sigma$-finite.

It is natural to consider the Painlev\'e problem for $K$-quasiregular mappings. Following \cite{astalaclopmateuorobitguriartepreprint}, we say that a compact set $E$ is {\it{removable for bounded $K$-quasiregular mappings}}, or simply {\it{$K$-removable}}, if for every open set $\Omega\supset E$, every $K$-quasiregular mapping $f:\Omega\setminus E\rightarrow\C$, with $f \in L^\infty(\Omega)$, admits a $K$-quasiregular extension to $\Omega$. In this definition, as in the analytic setting, we may replace $L^\infty(\Omega)$ by $BMO(\Omega)$ to get a close variant of the problem. We will refer to these two problems as {\it{$L^\infty$ $K$-removability}} and {\it{$BMO$ $K$-removability}}.

The critical dimension in both the $L^\infty$ and $BMO$ $K$-quasiregular removability problems is $\frac{2}{K+1}$. This is determined by the sharpness of the bounds in equation \eqref{distortionofdimension}. In fact, Iwaniec and Martin previously conjectured \cite{iwaniecmartinquasiregularevendimensions}  that in $\R^n$, $n \geq 2$, sets with Hausdorff measure  $\H^\frac{n}{K+1}(E)=0$ are removable for bounded $K$-quasiregular mappings. A positive answer for $n=2$ was described in \cite{astalaiwaniecmartin}. In \cite{astalaclopmateuorobitguriartepreprint} a stronger result is proven:

\begin{theorem}\label{theorems1.2and4.3inACMOU}\textbf{[Theorems 1.2 and 4.3 in \cite{astalaclopmateuorobitguriartepreprint}]}
Let $E$ be a compact set in the plane, and let $K>1$. Assume that $\H^\frac{2}{K+1}(E)$ is $\sigma$-finite. Then $E$ is removable for all bounded $K$-quasiregular mappings.

In particular, for any $K$-quasiconformal mapping $\phi$ the image $\phi(E)$ is purely unrectifiable.
\end{theorem}

Notice that the situation is somewhat different when $K = 1$, since for instance the line segment $E = [0,1]$  is not removable.

For  the converse direction towards showing that $\frac{2}{K+1}$ is the critical dimension in both the $L^\infty$ and $BMO$ $K$-quasiregular removability problems, Astala \cite{astalaareadistortion} found for every  $t> \frac{2}{K+1}$ non-$K$-removable sets with $\dim(E) = t$. In \cite{astalaclopmateuorobitguriartepreprint} the following is proven:

\begin{theorem}\label{theorem5.1inACMOU}\textbf{[Theorem 5.1 in \cite{astalaclopmateuorobitguriartepreprint}]}
There are compact sets with dimension precisely equal to $\frac{2}{K+1}$ yet not removable for some bounded $K$-quasiregular mappings.
\end{theorem}

The aforementioned Theorems \ref{theorem2.5inACMOU} (or Theorem \ref{theorem1.1inACMOU}) and \ref{theorems1.2and4.3inACMOU} are closely connected via the classical Stoilow factorization, which says (see \cite{astalaiwaniecmartin}, \cite{lehtovirtanen}), that in planar domains $K$-quasiregular mappings are precisely the maps $f = h \circ \phi$,  where $h$ is analytic and $\phi$ is $K$-quasiconformal. The idea in \cite{astalaclopmateuorobitguriartepreprint} is to combine distortion estimates for $\phi$ (such as Theorems \ref{theorem1.1inACMOU} or \ref{theorem2.5inACMOU}) and removability results for $h$ (i.e. analytic capacity results) in order to prove Theorem \ref{theorems1.2and4.3inACMOU}.

For the related BMO removability question, we have the following 

\begin{cor}\label{Corollary4.1inACMOU}\textbf{[Corollary 4.1 in \cite{astalaclopmateuorobitguriartepreprint}]}
Let $E$ be a compact subset of the plane. Assume that $\H^\frac{2}{K+1}(E)=0$. Then $E$ is removable for all $BMO$ $K$-quasiregular mappings.
\end{cor}

This follows immediately from Theorem \ref{theorem1.1inACMOU} and the aforementioned result by Kaufman \cite{kaufman} and Kr\'{a}l \cite{kral} that the condition $\H^1(E)=0$ is a precise characterization for removable singularities of  $BMO$ analytic functions. 

In light of the previous Theorems and Corollary, it is natural to wonder whether compact sets of sigma-finite $\frac{2}{K+1}$-Hausdorff measure are removable for BMO $K$-quasiregular mappings. Hence the question is raised (Question 4.2 in \cite{astalaclopmateuorobitguriartepreprint}) whether the BMO and the $L^{\infty}$ problems are indeed different, i.e. whether there exists for every $K>1$ a compact set $E$ of finite $\frac{2}{K+1}$-Hausdorff measure (hence $K$-removable, i.e. removable for the $L^{\infty}$ $K$-removability problem), which is not removable for BMO $K$-quasiregular mappings. This question was asked by Astala to the author before \cite{astalaclopmateuorobitguriartepreprint} in an equivalent form \cite{astalapersonalcommunication}.

One of our main results is the following 

\begin{theorem}\label{Question4.2inACMOUTHasPositiveAnswerIntroduction}
Question 4.2 in \cite{astalaclopmateuorobitguriartepreprint} has a positive answer. I.e. there exists for every $K \geq 1$ a compact set $E$ with $0 < \H^\frac{2}{K+1}(E) < \infty$, such that $E\,$  is not removable for some $K$-quasiregular functions in $BMO(\C)$. (The case $K=1$ is due to Kaufman and Kr\'{a}l, as we already mentioned.)
\end{theorem}

The construction that we give has implications to the important problem of determining whether \eqref{abscont} holds. To be more precise, any statement of the type $A \,\, \Longrightarrow \,\, B$ can be sharp in two different senses. Let us say that it has a ``sharp hypothesis" if any hypothesis $\widetilde{A}$ strictly weaker than $A$ cannot yield the same conclusion $B$, i.e. for any such $\widetilde{A}$ there are cases were $B$ is not satisfied. Let us say that it has a ``sharp conclusion" if any conclusion $\widetilde{B}$ strictly stronger than $B$ cannot hold under the hypothesis $A$, i.e. there are cases where $A$ is satisfied but $\widetilde{B}$ is not. Then another one of our main results is

\begin{thm}\label{SharpnessOfTheorems1.1inACMOUand2.5inACMOUIn2DifferentSensesIntroduction}
\begin{enumerate}
\item[(a)] Consider the statement \eqref{abscont}, i.e. that for any compact set $E \subset \C$ and any $K$-quasiconformal mapping $\phi:\C \to \C$, we have that $$\H^d(E)=0\,\,\,\Longrightarrow\,\,\,\H^{d'}(\phi(E))=0,$$
with $d'=\frac{2Kd}{2+(K-1)d}$ and $0<d<2$. If such a statement is true, then it is sharp in both the ``sharp hypothesis" and ``sharp conclusion" ways (provided the weaker hypothesis or conclusion are expressed in terms of Hausdorff gauge functions.)

Notice that \eqref{abscont} is true in the particular case $d=\frac{2}{K+1}$, $d'=1$, (Theorem \ref{theorem1.1inACMOU}), which is the relevant case for removability (Theorems \ref{theorems1.2and4.3inACMOU} and \ref{theorem5.1inACMOU}), and it is conjectured to be true for all $0<d<2$ (Question 4.4 in \cite{astalaareadistortion} and Conjecture 2.3 in \cite{astalaclopmateuorobitguriartepreprint}.)

\item[(b)] Consider the statement that for any compact set $E \subset \C$ and any $K$-quasiconformal mapping $\phi:\C \to \C$, we have that 
\begin{equation}\label{StatementSigmaFiniteGoesToSigmaFiniteInAllDimensionsIntroduction}
\H^d(E) \ is \ \sigma-finite \,\,\,\Longrightarrow\,\,\,\H^{d'}(\phi(E)) \ is \ \sigma-finite, 
\end{equation}
with $d'=\frac{2Kd}{2+(K-1)d}$ and $0<d<2$. If such a statement is true, then it is sharp in both the ``sharp hypothesis" and ``sharp conclusion" ways (provided the weaker hypothesis or conclusion are expressed in terms of Hausdorff gauge functions.)

Notice that \eqref{StatementSigmaFiniteGoesToSigmaFiniteInAllDimensionsIntroduction} is true in the particular case $d=\frac{2}{K+1}$, $d'=1$, (Theorem \ref{theorem2.5inACMOU}), which is the relevant case for removability, (Theorems \ref{theorems1.2and4.3inACMOU} and \ref{theorem5.1inACMOU}), and we conjecture it is true for all $0<d<2$.
\end{enumerate}
\end{thm}


We don't know if a set of finite Hausdorff measure $d$ is always mapped to a set of finite Hausdorff measure $d'$ (as opposed to $\sigma$-finite.)

Due to the previous comments on the implications for sharpness of the aforementioned Theorems, the constructions to be presented provide examples of ``maximum stretching" allowed by quasiconformal mappings at the fine level of Hausdorff measures. To understand better what we mean by ``maximum stretching", let us briefly review the history (to the best of our knowledge) of these ``maximum stretching" examples. Astala \cite{astalaareadistortion} showed that the dimension distortion bounds \eqref{distortionofdimension} are optimal by gluing a sequence of examples $E_n$ so that $\dim (E_n) = d$, and $\dim \phi(E_n) = d' - \varepsilon_n$, for a certain $K$-quasiconformal mapping $\phi$ and for a sequence $\varepsilon_n \to 0$. One can think that the ``string" joining the source and target sets was as close to the maximum stretching as possible. Later, Theorem \ref{theorem5.1inACMOU} in \cite{astalaclopmateuorobitguriartepreprint} gave an example of the string being at the maximum stretching, in the case of $d=\frac{2}{K+1}$ and $d'=1$ (although it works in other dimensions as well), but the maximum stretching was only in terms of dimension. I.e. \cite{astalaclopmateuorobitguriartepreprint} gave a $K$-quasiconformal mapping $\phi$ and a compact set $E$ so that $\dim (E) = d$, and $\dim \phi(E) = d'$, but this example did not give the maximum possible stretching at the finer level of Hausdorff measures. Whether this maximum possible stretching at the finer level of Hausdorff measures can occur remained open, and that is the content of Question 4.2 in \cite{astalaclopmateuorobitguriartepreprint} that we answer positively in this paper for all dimensions.


The Cantor-type construction that we will present is done with some radial stretchings on disks. Some packing problems appear which might suggest using similar radial stretchings on squares (with the $\ell^\infty$ norm instead of the Euclidean norm in $\C$), but then the constant of quasiconformality would be strictly larger than $K$. Hutchinson \cite{hutchinson} (see also \cite{mattila}) considered Cantor type sets where all the generations in the construction had the same number of children, and this number was a constant throughout the generations. A further generalization (see e.g. \cite{mattila}) makes this number increase very fast from one generation to the next. This was also the idea in Theorem \ref{theorem5.1inACMOU}. However in all these constructions, to the best of our knowledge, all children of a given ball had (roughly) the same size, and the children were uniformly distributed inside the father ball. 










The theorem in section 4.12 in \cite{mattila} (see also \cite{martinmattilakdimensional}) allows for more general constructions, but to our knowledge, those have not appeared previously in the literature.

However in this paper we need to construct a Cantor set where the children of the same ball are of very different sizes, and it was not clear a priori to us how to position the children inside the father ball, and what is the ``appropriate" thickness (size) for each of the children (indeed different children of the same father have different sizes) to yield a given Hausdorff measure for the resulting Cantor set. The previous constructions built Cantor sets where the children of a given father were roughly of the same size and uniformly distributed inside the father. So here an appropriate notion of ``uniformly distributed" and ``size" was needed. Here we construct such Cantor sets. A somewhat surprising aspect is that these appropriate notions of ``uniform distribution" and ``size" are naturally suggested by properties of $K$-quasiconformal mappings via an algebraic identity (see \eqref{EquationSigmaR} or the simpler case \eqref{EquationSigmaRInCaseTargetDimensionIs1}), which has a nice geometric interpretation in terms of area (see \eqref{SimplificationOfTargetEquationToArea}), which in turn is indeed very useful for proving the properties needed (see \eqref{ProofPackingTarget} and \eqref{ProofPackingTargetPart2}). This geometric interpretation becomes essentially the only guiding principle (even for the appropriate choice of definitions, see \eqref{SigmaAsFunctionOfRGeneralizedHausdorffMeasure}) for the proof of the general case, when technicalities are so pervasive that the algebraic intuition is lost. See also the comments after the proof of Lemma \ref{LemmaPackingTarget}.

The construction we present satisfies the hypotheses in the aforementioned Theorem 4.12 in \cite{mattila} (except for hypothesis (3) which is satisfied up to a factor of $(1-\varepsilon_{k+1})$, see \eqref{CoveringByBuildingBlocksConditionAtScaleNForSourceAndTarget}). However, we give the complete proof that the resulting sets have strictly positive and finite Hausdorff measure $\H^s$ for the convenience of the reader, to highlight the interactions between algebraic identities, quasiconformal mappings and geometry, and because we need also in section \ref{ConstructionForGeneralizedHausdorffMeasures} a result for more general gauge functions than the ones appearing in \cite{mattila}.



The construction we present is inspired in that of Theorem \ref{theorem5.1inACMOU} (which in turn is inspired in Theorem 18.7.1 in \cite{iwaniecmartin}), albeit a number of modifications and technical difficulties appear. 

We include in section \ref{BasicConstruction} only the basic construction of the Cantor-type set, leaving the choice of the key parameters for later sections. In section \ref{ConstructionForUsualHausdorffMeasures} the parameters are chosen to yield the proof for the case of the Hausdorff measures with gauge function $h(t)=t^\alpha$, i.e. the ``usual" Hausdorff measures. This easier case already contains most of the main ideas, and the considerably more technical general case is done in section \ref{ConstructionForGeneralizedHausdorffMeasures}.

\textit{Acknowledgements.} I would like to thank Stephen Montgomery-Smith who pointed out that for an increasing function $\varepsilon: \R^+ \to \R^+$ with $\varepsilon (0)=0$, the condition that $\frac{\varepsilon(t)}{t}$ is decreasing is very close to the condition that $\varepsilon(t)$ is concave, and that the decreasing condition for $\frac{\varepsilon(t)}{t}$ is indeed compatible with the logarithmic-type condition appearing in Theorem \ref{SharpExampleTheoremForGeneralizedHausdorffMeasures}. I would also like to thank K. Astala, A. Clop, J. Mateu and J. Orobitg for insightful conversations regarding the previous paper we wrote \cite{astalaclopmateuorobitguriartepreprint}. Further pondering on some of these insights inspired some of the ideas of this paper. I would also like to thank M. Jos\'{e} Mart\'{i}n for her help with the pictures.

In terms of acknowledgements not mathematically related to this paper, I would also like to thank my former advisor Peter W. Jones, my former mentor Kari Astala, my mentors Alex Iosevich and Igor Verbitsky; Mark Ashbaugh, Chris Bishop, Mario Bonk, Peter Casazza, David Drasin, the late Miguel de Guzm\'{a}n, Loukas Grafakos, Juha Heinonen, Eugenio Hern\'{a}ndez, Steve Hofmann, Tadeusz Iwaniec, Nigel Kalton, Nam-Gyu Kang, Carlos Kenig, Pekka Koskela, Nikolai Makarov, Olli Martio, Pertti Mattila, Anna Mazzucato, Dorina and Marius Mitrea, Paul F.X. M\"{u}ller, Tomi Nieminen, Jill Pipher, Luke Rogers, Eero Saksman, Raanan Schul, Hans-Olav Tylli, Joan Verdera and Dragan Vukotic for their support and help at some time or other (or continuously) during my postdoc positions at the Mathematics Department at the University of Missouri-Columbia and at Helsinki University, departments to which I am also indebted and grateful. I am omitting many others whom I keep in mind for reasons of space.

I would also like to thank my wife In\'{e}s and my family and friends for continued support.

\section{The basic construction}\label{BasicConstruction}

As we mentioned above, in \cite{astalaclopmateuorobitguriartepreprint}, the following question is asked:

\begin{quest}\label{question4.2inACMOU}\textbf{[Question 4.2 in \cite{astalaclopmateuorobitguriartepreprint}]}
Does there exist for every $K \geq 1$ a compact set $E$ with $0 < \H^\frac{2}{K+1}(E) < \infty$, such that $E\,$  is not removable for some $K$-quasiregular functions in $BMO(\C)$.
\end{quest}

Since the condition $\H^1(E)=0$ is a precise characterization for removable singularities of $BMO$ analytic functions, (\cite{kaufman}, \cite{kral}), the case $K=1$ is already known to have a positive answer.
Between the present section and section \ref{ConstructionForUsualHausdorffMeasures}, we will prove the following

\begin{thm}\label{TheoremConstructionForUsualHausdorffMeasures}
Let $K>1$. For any $0<t<2$, there exists a compact set $E$ with $0 < \H^t(E) < \infty$ and a $K$-quasiconformal mapping $\phi : \C \rightarrow \C$ such that $0 < \H^{t'}(\phi E) < \infty$, where $t'=\frac{2Kt}{2+(K-1)t}$ (see equations \eqref{abscont} and \eqref{distortionofdimension}.)

In particular, choosing $t=\frac{2}{K+1}$ (so $t'=1$) gives a positive answer to Question \ref{question4.2inACMOU}.
\end{thm}

\begin{proof}

We will construct the $K$-quasiconformal mapping $\phi$ as the limit of a sequence $\phi_N$ of $K$-quasiconformal mappings, and $E$ will be a Cantor-type set. To reach the optimal estimates we need to change, at every step in the construction of $E$, both the size and the number  $m_j$ of the generating disks. However, this change is made not only from one step to the next, as in \cite{astalaclopmateuorobitguriartepreprint}, but also within the same step of the construction.

As part of the motivation for the construction, notice that in the terminology of \cite{astalaclopmateuorobitguriartepreprint}, for the case $t=\frac{2}{K+1}$ and $t'=1$, we are (formally) choosing $\varepsilon (t) =1$. Then, the area (up to a multiplicative factor of $\pi$) covered by the $N^{th}$ generating disks is $c_N$, where $1>c_N=m_N\,R_N^2>\frac{1}{2}$, and $m_N$ is the number of generating disks which are chosen disjoint inside the unit disk. The product of the factors $c_N$ appears as part of the gauge function $\varepsilon ' (t)$ in \cite{astalaclopmateuorobitguriartepreprint}, which forces the distortion not to be as sharp as conceivably possible. So one would like the product of the areas covered at the different steps to be convergent to a strictly positive constant.

This observation motivates the following elementary Lemma, of which part (a) is well-known in the context of sphere packings:

\begin{lemma}\label{FillingAreaOfDiskWithDisks}
Let $\D = \{z \in \C : |z|<1 \}$.
\begin{enumerate}
\item[(a)] There exists an absolute constant $\varepsilon_0 >0$ such that for any $0<R<1$, and any collection of disks $D_j \subset \D$ with disjoint interiors, with radii $r_j=R$, $\mid \cup_j D_j  \mid < (1-\varepsilon_0) \mid \D \mid$, i.e. $\displaystyle \sum_j r_j^2 < 1-\varepsilon_0$. (Here $\mid A \mid$ is the area of $A$.)

\item[(b)] For any $\varepsilon >0$, there exists a finite collection of disks $D_j \subset \D$ with radii $r_j$ with disjoint interiors (or even disjoint closures), such that $\mid \cup_j D_j  \mid > (1-\varepsilon) \mid \D \mid$, i.e. $\displaystyle \sum_j r_j^2 > 1-\varepsilon$.
\end{enumerate}
\end{lemma}

\begin{proof}
Part (a) follows readily from the observation that given any 3 pairwise tangent disks $D_1, D_2, D_3$ with the same radius $R$, in the space they leave between them (i.e. in the bounded component of $\displaystyle \C \setminus \bigcup_{j=1}^{3} D_j$) one can fit another disk $B$, tangent to $D_1, D_2$ and $D_3$, with radius $cR$, where $c$ is an absolute constant independent of $R$.

Part (b) follows from Vitali's covering theorem, but we will prove it directly since we will later use some elements from the proof. Given a bounded open set $\Omega$, consider a mesh of squares of side $\delta$. Select those squares entirely contained in the open set, i.e. $\overline{Q_j} \subset \Omega$, say such a collection is $\{ Q_j  \}_{j=1}^{N}$. Then $\displaystyle \mid \Omega \setminus \bigcup_{j=1}^{N} Q_j \mid $ is as small as we wish if $\delta$ is sufficiently small.

For each $Q_j$, let $D_j$ be the largest disk inscribed inside it. (Shrink the $D_j$ slightly so that they have disjoint closures.) Then $\mid D_j \mid > \frac{1}{2} \mid Q_j \mid$.

Consequently, given $\Omega_0 = \D$, pick a first collection of disks $\ds \{D^1_j \}_{j=1}^{N}$ eating up at least, say, $\frac{1}{10}$ of the area of $\D$. Let $\ds \Omega_1 = \D \setminus \bigcup_{j=1}^{N} D^1_j$, which has area $< \frac{9}{10} \mid \Omega_0 \mid$. Repeat the construction in $\Omega_1$ and so on. The Lemma follows since $\left( \frac{9}{10} \right)^n \longrightarrow 0$ as $n \longrightarrow \infty$.

\end{proof}

Hence, by Lemma \ref{FillingAreaOfDiskWithDisks}, in order to fill a very big proportion of the area of the unit disk $\D$ with smaller disks we are forced to consider disks of different radii.
This creates a number of technical complications as we will see later.

{\bf{Step 1}}. Choose first $m_{1,1}$ disjoint disks $D(z_{1,1}^i,R_{1,1}) \subset \D$, $i=1,...,m_{1,1}$, and then $m_{1,2}$ disks $D(z_{1,2}^i,R_{1,2}) \subset \D$, $i=1,...,m_{1,2}$, disjoint among themselves and with the previous ones, and then $m_{1,3}$ disks $D(z_{1,3}^i,R_{1,3}) \subset \D$, $i=1,...,m_{1,3}$, disjoint among themselves and with the previous ones, and so on up to $m_{1,l_1}$ disks $D(z_{1,l_1}^i,R_{1,l_1}) \subset \D$, $i=1,...,m_{1,l_1}$, disjoint among themselves and with the previous ones,
so that they cover a big proportion of the unit disk $\D$ (see Lemma \ref{FillingAreaOfDiskWithDisks}.) Then, we have that
\begin{equation}\label{AreaCoveredInFirstStep}
c_1:=m_{1,1}\,(R_{1,1})^2 +  m_{1,2}\,(R_{1,2})^2 + ...+ m_{1,l_1}\,(R_{1,l_1})^2 = 1-\varepsilon_1 
\end{equation}

where $0< \varepsilon_1 <1$ is a very small parameter to be chosen later. By the proof of Lemma \ref{FillingAreaOfDiskWithDisks}, we can assume that all radii $R_{1,j} < \delta_1$, for $j=1, ..., l_1$, for a $\delta_1 >0$ as small as we wish.

Now consider the parameters $\sigma_{1,j} >0$, which we will associate to each one of the disks 
$D(z_{1,j}^i,R_{1,j})$, with $j=1, ..., l_1$, and all possible values of $i$. We associate the same parameter $\sigma_{1,j}$ to all the disks of the form $D(z_{1,j}^i,R_{1,j})$ (so $\sigma_{1,j}$ does not depend on $i$.) The parameters $\sigma_{1,j}$ will be chosen later, and they will all be quite small, say $\sigma_{1,j} < \frac{1}{100}$ for $j=1, ..., l_1$.

Next, let  $r_{1,j}=R_{1,j}$ for $j=1, ..., l_1$. For each $i=1,\dots, m_j$, let $\varphi^i_{1,j}(z)=z^i_{1,j}+(\sigma_{1,j})^K R_{1,j}\,z$ and, using the notation $\alpha D(z,\rho):= D(z,\alpha \rho)$, set
$$\aligned
D^i_j&:=\frac{1}{(\sigma_{1,j})^K}\,\varphi^i_{1,j}(\D)=D(z^i_{1,j}, r_{1,j})\\
(D^i_j)'&:=\varphi^i_{1,j}(\D)=D(z^i_{1,j},(\sigma_{1,j})^K r_{1,j}) \subset D^i_j
\endaligned$$

As the first approximation of the mapping define
$$
 g_1(z)=
\begin{cases}
(\sigma_{1,j})^{1-K}(z-z^i_{1,j})+z^i_{1,j}, &z\in (D^i_j)'\\
\left|\frac{z-z^i_{1,j}}{r_{1,j}}\right|^{\frac{1}{K}-1}(z-z^i_{1,j})+z^i_{1,j}, \; &z\in D^i_j\setminus (D^i_j)'\\
z, & z \notin \cup D^i_j
\end{cases}
$$
This is a $K$-quasiconformal mapping, conformal outside of $\ds \bigcup_{j=1}^{l_1} \bigcup_{i=1}^{m_{1,j}}(D^i_j \setminus (D^i_j)' )$. It maps each $D^i_j$ onto itself and $(D^i_j)'$ onto $(D^i_j)''=D(z^i_{1,j},\sigma_{1,j} \: r_{1,j})$, while the rest of the plane remains fixed. Write $\phi_1=g_1$. 
\\

{\bf{Step 2}}. We have already fixed $l_1, m_{1,j}, R_{1,j}, \sigma_{1,j}$ and $c_1$. Choose now $m_{2,1}$ disjoint disks $D(z_{2,1}^n,R_{2,1}) \subset \D$, $n=1,...,m_{2,1}$, and then $m_{2,2}$ disks $D(z_{2,2}^n,R_{2,2}) \subset \D$, $n=1,...,m_{2,2}$, disjoint among themselves and with the previous ones (within this second step), and then $m_{2,3}$ disks $D(z_{2,3}^n,R_{2,3}) \subset \D$, $n=1,...,m_{2,3}$, disjoint among themselves and with the previous ones (within this second step), and so on up to $m_{2,l_2}$ disks $D(z_{2,l_2}^n,R_{2,l_2}) \subset \D$, $n=1,...,m_{2,l_2}$, disjoint among themselves and with the previous ones (within this second step), so that they cover a big proportion of the unit disk $\D$ (see Lemma \ref{FillingAreaOfDiskWithDisks}.) Then, we have that
\begin{equation}\label{AreaCoveredInSecondStep}
c_2:=m_{2,1}\,(R_{2,1})^2 +  m_{2,2}\,(R_{2,2})^2 + ...+ m_{2,l_2}\,(R_{2,l_2})^2 = 1-\varepsilon_2 
\end{equation}

where $0< \varepsilon_2 <1$ is a very small parameter to be chosen later. By the proof of Lemma \ref{FillingAreaOfDiskWithDisks}, we can assume that all radii $R_{2,k} < \delta_2$, for $k=1, ..., l_2$, for a $\delta_2 >0$ as small as we wish.

Repeating the above procedure, consider now the parameters $\sigma_{2,k} >0$, which we will associate to each one of the disks $D(z_{2,k}^n,R_{2,k})$, with $k=1, ..., l_2$, and all possible values of $n$. We associate the same parameter $\sigma_{2,k}$ to all the disks of the form $D(z_{2,k}^n,R_{2,k})$ (which is why the parameter $\sigma_{2,k}$ does not have an index depending on $n$.) The parameters $\sigma_{2,k}$ will be chosen later, and they will all be quite small, say $\sigma_{2,k} < \frac{1}{100}$ for $k=1, ..., l_2$.

Denote $r_{\{2,k\},\{1,j\}}=R_{2,k}\,\sigma_{1,j} \: r_{1,j}$ and $\varphi^n_{2,k}(z)=z^n_{2,k}+(\sigma_{2,k})^K R_{2,k}\,\,z$, \, and define the auxiliary disks
$$\aligned
D_{j,k}^{i,n}=\phi_1\left(\frac{1}{(\sigma_{2,k})^K}\, \varphi^{i}_{1,j} \circ \varphi^{n}_{2,k}(\D)\right)=D(z^{i,n}_{j,k}, r_{\{2,k\},\{1,j\}})\\
(D_{j,k}^{i,n})'=\phi_1\left(\, \varphi^{i}_{1,j} \circ \varphi^{n}_{2,k}(\D)\right)=D(z^{i,n}_{j,k}\, , (\sigma_{2,k})^K r_{\{2,k\},\{1,j\}})
\endaligned$$
for certain $z^{i,n}_{j,k} \in \D$, where $i=1,\dots,m_{1,j}$, $n=1,\dots,m_{2,k}$, $j=1,\dots,l_1$ and $k=1,\dots,l_2$. Now let

$$
g_2(z)=
\begin{cases}
(\sigma_{2,k})^{1-K}(z-z^{i,n}_{j,k})+z^{i,n}_{j,k}&z\in (D_{j,k}^{i,n})'\\
\left|\frac{z-z^{i,n}_{j,k}}{r_{\{2,k\},\{1,j\}}}\right|^{\frac{1}{K}-1}(z-z^{i,n}_{j,k})+z^{i,n}_{j,k}&z\in D_{j,k}^{i,n}\setminus (D_{j,k}^{i,n})'\\
z&\text{otherwise}
\end{cases}
$$

Clearly, $g_2$ is $K$-quasiconformal, conformal outside of $\ds \bigcup_{i,j,k,n} \left( D_{j,k}^{i,n} \setminus (D_{j,k}^{i,n})' \right)$, maps each $D_{j,k}^{i,n}$ onto itself and $(D_{j,k}^{i,n})'$ onto $(D_{j,k}^{i,n})''=D(z^{i,n}_{j,k},\: \sigma_{2,k}\: r_{\{2,k\},\{1,j\}})$, while the rest of the plane remains fixed. 
Define $\phi_2=g_2\circ\phi_1$.\\

The picture below represents (an approximation of) the $K$-quasiconformal mapping $\phi$ by its first two steps (i.e. by $\phi_2$.) The size of the parameters $\sigma$ has been greatly magnified for the convenience of the reader (so that e.g. the annuli $D_{j}^{i}\setminus (D_{j}^{i})'$ and their images under $\phi$ are much thinner in the picture than in the proof.)



\begin{figure}[ht]
\begin{center}
\includegraphics{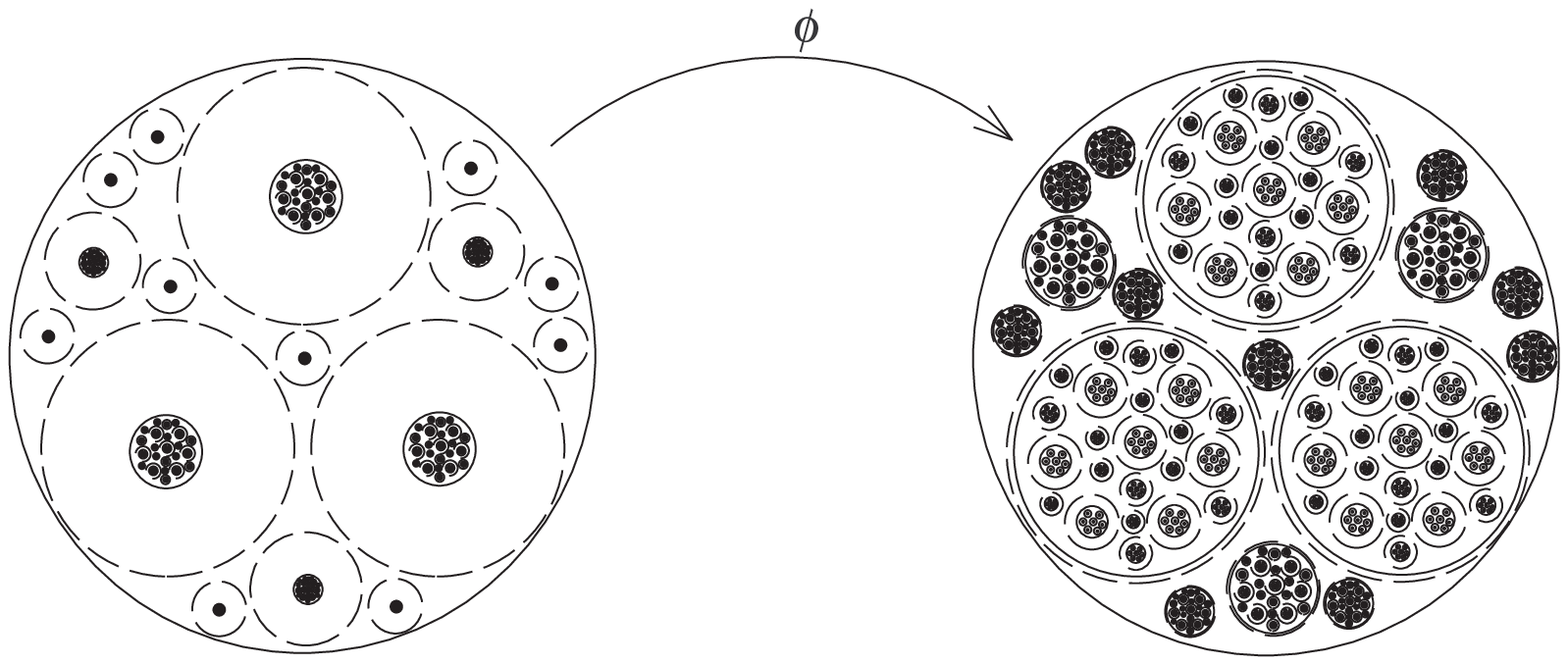}
\end{center}
\end{figure}

\noindent  {\bf{The induction step}}. After step $N-1$ we take $m_{N,1}$ disjoint disks $D(z_{N,1}^q,R_{N,1}) \subset \D$, $q=1,...,m_{N,1}$, and then $m_{N,2}$ disks $D(z_{N,2}^q,R_{N,2}) \subset \D$, $q=1,...,m_{N,2}$, disjoint among themselves and with the previous ones (within this $N^{th}$ step), and then $m_{N,3}$ disks $D(z_{N,3}^q,R_{N,3}) \subset \D$, $q=1,...,m_{N,3}$, disjoint among themselves and with the previous ones (within this $N^{th}$ step), and so on up to $m_{N,l_N}$ disks $D(z_{N,l_N}^q,R_{N,l_N}) \subset \D$, $q=1,...,m_{N,l_N}$, disjoint among themselves and with the previous ones (within this $N^{th}$ step), so that they cover a big proportion of the unit disk $\D$ (see Lemma \ref{FillingAreaOfDiskWithDisks}.) Then, we have that
\begin{equation}\label{AreaCoveredInNthStep}
c_N:=m_{N,1}\,(R_{N,1})^2 +  m_{N,2}\,(R_{N,2})^2 + ...+ m_{N,l_N}\,(R_{N,l_N})^2 = 1-\varepsilon_N 
\end{equation}

where $0< \varepsilon_N <1$ is a very small parameter to be chosen later. By the proof of Lemma \ref{FillingAreaOfDiskWithDisks}, we can assume that all the radii $R_{N,p} < \delta_N$, for $p=1, ..., l_N$, and for a $\delta_N >0$ as small as we wish.

Repeating the above procedure, consider now the parameters $\sigma_{N,p} >0$, which we will associate to each one of the disks $D(z_{N,p}^q,R_{N,p})$, with $p=1, ..., l_N$, and all possible values of $q$. We associate the same parameter $\sigma_{N,p}$ to all the disks of the form $D(z_{N,p}^q,R_{N,p})$ (so the parameter $\sigma_{N,p}$ does not depend on $q$.) The parameters $\sigma_{N,p}$ will be chosen later, and they will all be quite small, say $\sigma_{N,p} < \frac{1}{100}$ for $p=1, ..., l_N$.

Denote then 
$r_{ \{N,p\}, \{N-1,h\}, \ldots , \{2,k\}, \{1,j\} }=R_{N,p}\,\,\sigma_{N-1,h}\,\, 
r_{ \{N-1,h\}, \ldots , \{2,k\}, \{1,j\} }$, and $\varphi^q_{N,p}(z)=z^q_{N,p}+ (\sigma_{N,p})^K \, R_{N,p}\,z$. For any multiindexes $I=(i_1,...,i_N)$ and  $J=(j_1,...,j_N)$, where $1\leq i_k\leq m_{k,j_k}$, $1\leq j_k\leq l_k$, and $k=1,...,N$, let 

\begin{equation}\label{FormulaDIJAndDIJPrime}
\aligned
D^{I}_{J}=\phi_{N-1}\left(\frac{1}{(\sigma_{N,p})^K}\, \varphi^{i_1}_{1,j_1} \circ \dots \circ \varphi^{i_N}_{N,j_N}(\D) \right) = D\left(z^{I}_{J}, r_{ \{N,p\}, \{N-1,h\}, \ldots , \{2,k\}, \{1,j\} } \right)\\
(D^{I}_{J})'=\phi_{N-1}\left( \varphi^{i_1}_{1,j_1}  \circ \dots \circ \varphi^{i_N}_{N,j_N}(\D) \right) = D\left(z^{I}_{J}, (\sigma_{N,p})^K \, r_{ \{N,p\}, \{N-1,h\}, \ldots , \{2,k\}, \{1,j\} }\right)
\endaligned
\end{equation}
and let

$$
g_N(z)=
\begin{cases}
(\sigma_{N,p})^{1-K}(z-z^{I}_{J})+z^{I}_{J}&z\in (D^{I}_{J})'\\
\left|\frac{z-z^{I}_{J}}{r_{ \{N,p\}, \{N-1,h\}, \ldots , \{2,k\}, \{1,j\} }}\right|^{\frac{1}{K}-1}(z-z^{I}_{J})+z^{I}_{J}&z\in D^{I}_{J}\setminus (D^{I}_{J})'\\
z&\text{otherwise}
\end{cases}
$$

Clearly, $g_N$ is $K$-quasiconformal, conformal outside of $\ds \bigcup_{\substack{I=(i_1,...,i_N)\\J=(j_1,...,j_N)}} \left( \, D^{I}_{J}\setminus (D^{I}_{J})' \, \right)$, maps $D^{I}_{J}$ onto itself and $(D^{I}_{J})'$ onto $(D^{I}_{J})''=D \left( z^{I}_{J},\, \sigma_{N,p} \, \, r_{ \{N,p\}, \{N-1,h\}, \ldots , \{2,k\}, \{1,j\} } \right)$, while the rest of the plane remains fixed. 
Now define $\phi_N=g_N\circ\phi_{N-1}$.\\  
\\
Since each $\phi_N$ is $K$-quasiconformal and  equals the identity  outside the unit disk $\D$, there exists a limit $K$-quasiconformal mapping  
$$\phi=\lim_{N\to\infty}\phi_N$$
with convergence in $W^{1,p}_{loc}(\C)$ for any $p<\frac{2K}{K-1}$.

On the other hand, $\phi$ maps the compact set $$E=\displaystyle\bigcap_{N=1}^\infty\left(
\ds \bigcup_{\substack{i_1,...,i_N \\ j_1,...,j_N}}
\varphi^{i_1}_{1,j_1}  \circ \dots \circ \varphi^{i_N}_{N,j_N} \left( \, \overline{\D} \, \right)
\right)$$
to the compact set
$$\phi(E)=\bigcap_{N=1}^\infty\left(
\ds \bigcup_{\substack{i_1,...,i_N \\ j_1,...,j_N}}
\psi^{i_1}_{1,j_1}  \circ \dots \circ \psi^{i_N}_{N,j_N} \left( \, \overline{\D} \, \right)
\right)$$
where we have written $\psi^{i_k}_{k,j_k}(z)=z^{i_k}_{k,j_k} + \sigma_{k,j_k} \, R_{k,j_k} \, z$, and where $1\leq i_k\leq m_{k,j_k}$, $1\leq j_k\leq l_k$, and $k \in \N$. 

Notice that with our notation, a building block in the $N^{th}$ step of the construction in the source set $E$ (i.e. a set of the type $\varphi^{i_1}_{1,j_1}  \circ \dots \circ \varphi^{i_N}_{N,j_N} \left( \, \overline{\D} \, \right)$) is a disk with radius given by 
\begin{equation}\label{RadiusSourceNthStep}
s_{j_1,...,j_N}=\left( (\sigma_{1,j_1})^K \, R_{1,j_1} \right) \dots \left( (\sigma_{N,j_N})^K R_{N,j_N} \right) 
\end{equation}

and a building block in the $N^{th}$ step of the construction in the target set $\phi(E)$ (i.e. a set of the type $\psi^{i_1}_{1,j_1}  \circ \dots \circ \psi^{i_N}_{N,j_N} \left( \, \overline{\D} \, \right)$) is a disk with radius given by 
\begin{equation}\label{RadiusTargetNthStep}
t_{j_1,...,j_N}=\left( \sigma_{1,j_1} \, R_{1,j_1} \right) \dots \left( \sigma_{N,j_N} \, R_{N,j_N} \right)
\end{equation}

\section{Examples of extremal distortion for Hausdorff measures}\label{ConstructionForUsualHausdorffMeasures}

In this section we prove Theorem \ref{TheoremConstructionForUsualHausdorffMeasures}, i.e. we deal with Hausdorff measures $\H^\alpha$ with gauge function $h(t)=t^\alpha$, where $0<\alpha<2$. We want to choose the parameters from section \ref{BasicConstruction} to yield a compact set $E$ such that $0< \H^d(E) < \infty$ and $0< \H^{d'}(\phi(E)) < \infty$, where $d'=\frac{2Kd}{2+(K-1)d}$ (see equations \eqref{abscont} and \eqref{distortionofdimension}.)

On the first step of the construction, we have the equation corresponding to the proportion of area taken by the disks in the first step, which is equation \eqref{AreaCoveredInFirstStep} that we repeat here for the convenience of the reader:
\begin{equation}\label{AreaCoveredInStep1}
c_1:=m_{1,1}\,(R_{1,1})^2 +  m_{1,2}\,(R_{1,2})^2 + ...+ m_{1,l_1}\,(R_{1,l_1})^2 = 1-\varepsilon_1 
\end{equation}

Since we aim at $\H^{d}(E)\approx 1$, (see equation \eqref{RadiusSourceNthStep}) it would be desirable to have
\begin{equation}\label{SourceEquationInStep1}
m_{1,1}\,\left[ (\sigma_{1,1})^K \, R_{1,1} \right]^d +  m_{1,2}\,\left[ (\sigma_{1,2})^K \, R_{1,2} \right]^d  + ...+ m_{1,l_1}\,\left[ (\sigma_{1,l_1})^K \, R_{1,l_1} \right]^d \approx 1 
\end{equation}

and in order to have $\H^{d'}(E)\approx 1$, (see equation \eqref{RadiusTargetNthStep}) it would also be desirable to have
\begin{equation}\label{TargetEquationInStep1}
m_{1,1}\,\left[ \sigma_{1,1} \, R_{1,1} \right]^{d'} +  m_{1,2}\,\left[ \sigma_{1,2} \, R_{1,2} \right]^{d'}  + ...+ m_{1,l_1}\,\left[ \sigma_{1,l_1} \, R_{1,l_1} \right]^{d'}  \approx 1 
\end{equation}

In general, for the $N^{th}$ step of the construction, we have the equation corresponding to the proportion of area taken by the disks in the $N^{th}$ step, which is equation \eqref{AreaCoveredInNthStep} that we repeat here for the convenience of the reader:
\begin{equation}\label{AreaCoveredInStepN}
c_N:=m_{N,1}\,(R_{N,1})^2 +  m_{N,2}\,(R_{N,2})^2 + ...+ m_{N,l_N}\,(R_{N,l_N})^2 = 1-\varepsilon_N 
\end{equation}

Analogously, it would be desirable to have (regarding the source set) 
\begin{equation}\label{SourceEquationInStepN}
m_{N,1}\,\left[ (\sigma_{N,1})^K \, R_{N,1} \right]^d +  m_{N,2}\,\left[ (\sigma_{N,2})^K \, R_{N,2} \right]^d  + ...+ m_{N,l_N}\,\left[ (\sigma_{N,l_N})^K \, R_{N,l_N} \right]^d \approx 1 
\end{equation}

and, regarding the target set,
\begin{equation}\label{TargetEquationInStepN}
m_{N,1}\,\left[ \sigma_{N,1} \, R_{N,1} \right]^{d'} +  m_{N,2}\,\left[ \sigma_{N,2} \, R_{N,2} \right]^{d'}  + ...+ m_{N,l_N}\,\left[ \sigma_{N,l_N} \, R_{N,l_N} \right]^{d'}  \approx 1 
\end{equation}

The Cantor-type sets $E$ and $\phi(E)$ are not self-similar (since on each step we introduce a different number and configuration of disks.) On top of that, the disks introduced at each step are of very different sizes among themselves, and their centers are by far non-uniformly distributed inside $\D$. Hence the technical difficulties to compute Hausdorff measures or dimension of these sets are, in general, quite substantial (at least to our knowledge.) The equations that appear for the source and target sets at scale $N$ are:
\begin{equation}\label{GeneralEquationSourceSetAtScaleN}
\sum_{j_1,...,j_N}  m_{1,j_1} m_{2,j_2} \ldots m_{N,j_N} \left( s_{j_1,...,j_N} \right)^d \approx 1
\end{equation}

(see equation \eqref{RadiusSourceNthStep}), and (see equation \eqref{RadiusTargetNthStep})
\begin{equation}\label{GeneralEquationTargetSetAtScaleN}
\sum_{j_1,...,j_N}  m_{1,j_1} m_{2,j_2} \ldots m_{N,j_N} \left( t_{j_1,...,j_N} \right)^{d'} \approx 1
\end{equation}

One of the merits of this paper is to present one way of handling such Cantor-type sets.

The first step is to notice, when comparing equations \eqref{SourceEquationInStep1}  and \eqref{TargetEquationInStep1} (or equations \eqref{SourceEquationInStepN} and \eqref{TargetEquationInStepN}), that for any two numbers  $\sigma , R >0$, the following identity holds (recall $d' = \frac{2Kd}{2+(K-1)d} $):
\begin{equation}\label{EquationSigmaR}
\left( \sigma^K \, R \right)^d = 
\left( \sigma \, R \right)^{ \frac{2Kd}{2+(K-1)d} } \left(  \frac{ \sigma^{dK} }{ R^{2-d} } \right)^{ \frac{d(K-1)}{2+(K-1)d} }
\end{equation}

This identity, which follows from elementary calculations, suggests the choice of parameters 
\begin{equation}\label{SigmaAsFunctionOfR}
(\sigma_{k,j_k})^{dK} = (R_{k,j_k})^{2-d}
\end{equation}

for all possible values of $k$ and $j_k$. This choice of parameters makes the left-hand sides of equations \eqref{SourceEquationInStepN} and \eqref{TargetEquationInStepN} equal. This algebraic identity is another key idea in this paper. It was originally seen in the simpler but important case $d=\frac{2}{K+1}$ and $d' =1$, where the following simplified identity holds:
\begin{equation}\label{EquationSigmaRInCaseTargetDimensionIs1}
\left( \sigma^K \, R \right)^{\frac{2}{K+1}} = \sigma \, R \left(  \frac{ \sigma }{ R } \right)^{ \frac{K-1}{K+1} }
\end{equation}

which suggests the choice $\sigma = R$ in that case (which is consistent with \eqref{SigmaAsFunctionOfR}.)

Moreover, the choice \eqref{SigmaAsFunctionOfR} 
actually has some geometric meaning. Namely, 
\begin{equation}\label{SimplificationOfTargetEquationToArea}
\left( \sigma^K \, R \right)^d = \left( \sigma \, R \right)^{ \frac{2Kd}{2+(K-1)d} } = \left( \sigma \, R \right)^{d'} = R^2
\end{equation}

i.e. the left-hand side of equation \eqref{TargetEquationInStepN} equals the left-hand side of equation \eqref{AreaCoveredInStepN}, which has a clear geometric interpretation in terms of area, already mentioned. This geometric interpretation is another key idea of the paper and it will prove very useful later.

Summarizing, with this choice of parameters, (see equations \eqref{AreaCoveredInStepN}, \eqref{SourceEquationInStepN} and \eqref{TargetEquationInStepN}) we have 

\begin{eqnarray}\label{AreaSourceAndTargetConditionsBecomeTheSameAtScaleN}
c_N &=& m_{N,1}\,(R_{N,1})^2 +  m_{N,2}\,(R_{N,2})^2 + ...+ m_{N,l_N}\,(R_{N,l_N})^2 = \nonumber \\
&=& m_{N,1}\,\left[ (\sigma_{N,1})^K \, R_{N,1} \right]^d +  m_{N,2}\,\left[ (\sigma_{N,2})^K \, R_{N,2} \right]^d  + ...+ m_{N,l_N}\,\left[ (\sigma_{N,l_N})^K \, R_{N,l_N} \right]^d = \nonumber \\
&=& m_{N,1}\,\left[ \sigma_{N,1} \, R_{N,1} \right]^{d'} +  m_{N,2}\,\left[ \sigma_{N,2} \, R_{N,2} \right]^{d'}  + ...+ m_{N,l_N}\,\left[ \sigma_{N,l_N} \, R_{N,l_N} \right]^{d'} = \nonumber \\
&=& 1-\varepsilon_N 
\end{eqnarray}

Then we have that (see equations \eqref{GeneralEquationSourceSetAtScaleN} and \eqref{GeneralEquationTargetSetAtScaleN})
\begin{equation}\label{CoveringByBuildingBlocksConditionAtScaleNForSourceAndTarget}
\sum_{j_1,...,j_N}  m_{1,j_1} m_{2,j_2} \ldots m_{N,j_N} \left( s_{j_1,...,j_N} \right)^d = 
\sum_{j_1,...,j_N}  m_{1,j_1} m_{2,j_2} \ldots m_{N,j_N} \left( t_{j_1,...,j_N} \right)^{d'} = 
\prod_{n=1}^{N} \left( 1- \varepsilon_n \right) 
\end{equation}

Notice that the product term in \eqref{CoveringByBuildingBlocksConditionAtScaleNForSourceAndTarget} is the proportion of area of $\D$ occupied by the (dilated and translated versions of the) disks $D(z_{N,p}^q,R_{N,p})$, with $p=1, ..., l_N$, and $1 \leq q \leq m_{N,p}$ so that they are placed inside the corresponding disks of all previous steps in section \ref{BasicConstruction} (see above \eqref{AreaCoveredInNthStep}.) This geometric interpretation will be useful later and provides for an essentially automatic way to check condition (3) in Theorem 4.12 in \cite{mattila}.

Now take $\varepsilon_n \rightarrow 0$ so fast that 
\begin{equation}\label{ProductOfAreasConverges}
\prod_{n=1}^{\infty} \left( 1- \varepsilon_n \right) \approx 1
\end{equation}

Such a choice of $\sigma 's$ and $\varepsilon_n$ will make equations \eqref{GeneralEquationSourceSetAtScaleN} and \eqref{GeneralEquationTargetSetAtScaleN} and an area condition true at the same time, as we will see later.

As a consequence, since the $R_{k,j_k}$ can all be taken small enough that all $\sigma_{k,j_k}<1$ and that all $R_{k,j_k}<\frac{1}{2}$, then
$\diam(\varphi^{i_1}_{1,j_1}  \circ \dots \circ \varphi^{i_N}_{N,j_N} \left( \, \overline{\D} \, \right)) = 2 \, s_{j_1,...,j_N}  \leq 2\, \left( \frac{1}{2} \right)^N \to 0$ when $N\to \infty$, and we have by \eqref{GeneralEquationSourceSetAtScaleN}, 
\begin{eqnarray}\label{UpperHausdorffMeasureEstimateForSource}
\H^d(E)=\lim_{\delta\to 0}\H^{h}_{\delta}(E)\leq 
\lim_{N\to \infty} 
\sum_{\substack{i_1,\dots,i_{N} \\ j_1,\dots,j_{N} }}
\left[ \diam(\varphi^{i_1}_{1,j_1}  \circ \dots \circ \varphi^{i_N}_{N,j_N} \left( \, \overline{\D} \, \right)) \right]^d \approx   \nonumber \\
\approx \lim_{N\to \infty}  \sum_{j_1,...,j_N}  m_{1,j_1} m_{2,j_2} \ldots m_{N,j_N} \left( s_{j_1,...,j_N} \right)^d 
=\prod_{n=1}^{\infty} \left( 1- \varepsilon_n \right)
\approx 1
\end{eqnarray}

A similar argument based on \eqref{GeneralEquationTargetSetAtScaleN}, yields $\H^{d'}(\phi(E))\lesssim 1$. We have established the following 

\begin{lemma}\label{UpperBoundsForHausdorffMeasureInSourceAndTarget}
With the notation as above, we have that
$$\H^{d}(E)\lesssim 1 \, \, \text{ and that } \, \, 
\H^{d'}(\phi(E))\lesssim 1. $$
\end{lemma}

As usual, the lower estimates for Hausdorff measures take some more work to establish than the upper estimates. Let us this time work with the target set $\phi(E)$. 

Fix a building block $D$ at scale $N-1$ for the target set, i.e. let $D = \psi^{i_1}_{1,j_1}  \circ \dots \circ \psi^{i_{N-1}}_{N-1,j_{N-1}} \left( \, \overline{\D} \, \right)$ for some choice of $i_k$ and $j_k$, $1 \leq k \leq N-1$. As is traditional, we will call the building blocks at scale $N$ contained in $D$, the children of $D$. I.e. the children of $D$ are the disks of the form $B = \psi^{i_1}_{1,j_1}  \circ \dots \circ \psi^{i_{N-1}}_{N-1,j_{N-1}} \circ \psi^{i_{N}}_{N,j_{N}}  \left( \, \overline{\D} \, \right)$, for some choice of $i_{N}$ and $j_{N}$, but with the same choices of $i_k$ and $j_k$ for $1 \leq k \leq N-1$ as for $D$. The genealogical terminology (parents, cousins, descendants, generation, etc.) has the obvious meaning in this context.

For a given building block $B$, let us denote by $r(B)$ its radius. Observe that, with the above notation, by \eqref{SimplificationOfTargetEquationToArea} and \eqref{AreaCoveredInStepN},
\begin{eqnarray}\label{SummingInAllChildrenOfD}
\sum_{B_n \, \text{ children of } \,D} r(B_n)^{d'} &=& 
%
\left[ \sigma_{1,j_1} \, R_{1,j_1} \dots 
\sigma_{N-1,j_{N-1}} \, R_{N-1,j_{N-1}} \right]^{d'} \, 
\sum_{j_N} m_{N,j_N} \left( \sigma_{N,j_N} \, R_{N,j_N}  \right)^{d'} = \nonumber \\
&=& r(D)^{d'} \, \sum_{j_N} m_{N,j_N} \left( R_{N,j_N} \right)^2 = 
r(D)^{d'} \, (1-\varepsilon_N) 
\end{eqnarray}

Consequently, if $D$ is a fixed building block at scale $L$, and we fix a finite family of building blocks $\{B_n \}$ with $B_n \subset D$, (the $B_n$ need not all be in the same generation), and if $\G(B_n)$ denotes the generation $B_n$ belongs to (i.e. the step in the construction in which it appears), assume $\max \G(B_n) = N$. Let $\{B_{N,k} \}$ be the collection of all descendants of elements of $\{B_n \}$ of generation $N$. Then, by \eqref{SummingInAllChildrenOfD}, we get
\begin{equation}\label{SummingInAllNGenerationDescendantsOfASubfamilyOfD}
\sum_{B_n} r(B_n)^{d'} \approx \sum_{B_{N,k}} r(B_{N,k})^{d'}
\end{equation}

where the comparability constants can be taken independent of $N$, i.e. they can be taken to be $\ds C = \prod_{n=1}^{\infty} \left( 1- \varepsilon_n \right)$ and $\frac{1}{C}$.
\\

We are aiming at the following

\begin{lemma}\label{LemmaPackingTarget}
Let $B$ be an arbitrary disk and $B_n$ be disjoint building blocks for $\phi(E)$, i.e. disks of the form $B_n = \psi^{i_1}_{1,j_1}  \circ \dots \circ \psi^{i_{N_n-1}}_{N_n-1,j_{N_n-1}} \circ \psi^{i_{N_n}}_{N_n,j_{N_n}}  \left( \, \overline{\D} \, \right)$, for some choice of indexes $N_n$, $i_k$, and $j_k$, for $1 \leq k \leq N_n$.

Let $\mathcal{C} = \{ B_n \}$ be a family of such building blocks $B_n$. We say $\mathcal{C} $ is admissible for $B$ (denoted by $\mathcal{C} \in \mathcal{A}(B)$) if all elements $B_n$ of $\mathcal{C} $ are pairwise disjoint and satisfy $B_n \subset B$.

There exists an absolute constant $C_1$ such that if $\mathcal{C} \in \mathcal{A}(B)$, then we have the following Carleson packing condition

\begin{equation}\label{PackingTarget}
\sum_{B_n \in \mathcal{C} } r(B_n)^{d'} \leq C_1 \, r(B)^{d'}.
\end{equation}\\
\end{lemma}

Precisely \eqref{PackingTarget} is the key step in proving the lower bound for the Hausdorff measure estimate, as is to be expected.

\begin{proof} (Of Lemma \ref{LemmaPackingTarget}.)


Let us fix a family $\mathcal{C} = \{ B_n \} \in \mathcal{A}(B)$. We can assume without loss of generality that $\mathcal{C}$ is a finite family. (Indeed, $\C$ is separable, so $\mathcal{C}$ contains at most countably many elements. Proving \eqref{PackingTarget} with the same constant $C_1$ for any finite subcollection of elements of $\mathcal{C}$ yields \eqref{PackingTarget} for $\mathcal{C}$.)

An iteration of \eqref{SummingInAllChildrenOfD} shows that \eqref{PackingTarget} holds for $B = \D$, and hence, we can assume without loss of generality that $r(B) \lesssim 1$. Now fix a disk $B$ and a maximal family of building blocks $\{B_n\}$ with $B_n \subset B$. Let $H$ be the unique integer with the property that there exists a building block disk $B^{H-1}_{i_0}$ of generation $\G (B^{H-1}_{i_0}) = H-1$, (i.e. $B^{H-1}_{i_0} = \psi^{i_1}_{1,j_1}  \circ \dots \circ \psi^{i_{H-1}}_{H-1,j_{H-1}} (\overline{\D})$ for some choice of $i_1, \dots ,i_{H-1}$ and of $j_1, \dots ,j_{H-1}$) such that $B_n \subset B^{H-1}_{i_0}$ for all $n$, but there is no building block disk $B^H_{j_0}$ of generation $\G ( B^H_{j_0} ) = H$ such that $B_n \subset B^{H}_{j_0}$ for all $n$. I.e. all the $B_n$ are descendants of some siblings $B^H_{k_0}, B^H_{k_1}, \dots , B^H_{k_m}$, with $m \geq 1$ (i.e. at least there are two siblings), which have a common father $B^{H-1}_{i_0}$. We assume that $\{ B^H_{k_p} \}_{p=0}^{m}$ is a complete list of ancestors of generation $H$ of the family $\{ B_n \}$.

Again an iteration of \eqref{SummingInAllChildrenOfD} shows that \eqref{PackingTarget} holds for $B = B^{H-1}_{i_0}$, and hence, we can assume without loss of generality that 
\begin{equation}\label{wlogRBLessThanRBHMinus1}
r(B) \lesssim r(B^{H-1}_{i_0}) = t_{j_1,...,j_{H-1}}= \left( \sigma_{1,j_1} \, R_{1,j_1} \right) \dots \left( \sigma_{H-1,j_{H-1}} \, R_{H-1,j_{H-1}} \right).
\end{equation}

Each $B^H_{k_p}$, $p=0,1, \dots , m$, has radius 
\begin{equation}\label{RadiusBHkp}
r(B^H_{k_p})= t_{j_1,...,j_{H-1},j_H}=\left( \sigma_{1,j_1} \, R_{1,j_1} \right) \dots \left( \sigma_{H-1,j_{H-1}} \, R_{H-1,j_{H-1}} \right) \left( \sigma_{H,j_H} \, R_{H,j_H} \right),
\end{equation}
for some choice of $j_1,...,j_{H-1},j_H$. To make the dependence on $p$ more apparent (notice that it only appears in the $H^{th}$ index, since the first $H-1$ entries are the same for all $p$ and they are determined by $B^{H-1}_{i_0}$), we will denote 
\begin{equation}\label{RadiusBHkpAsFunctionOfP}
r(B^H_{k_p})= t_{j_1,...,j_{H-1},j_{H_{k_p}}}=\left( \sigma_{1,j_1} \, R_{1,j_1} \right) \dots \left( \sigma_{H-1,j_{H-1}} \, R_{H-1,j_{H-1}} \right) \left( \sigma_{H,j_{H_{k_p}}} \, R_{H,j_{H_{k_p}}} \right).
\end{equation}

Associated to each $B^H_{k_p}$, consider the disk $\widetilde{B^H_{k_p}}$, which is concentric to $B^H_{k_p}$, and has radius 
\begin{eqnarray}\label{RadiusWideTildeBHkp}
r(\widetilde{B^H_{k_p}}) &=& \frac{t_{j_1,...,j_{H-1},j_H}}{\sigma_{H,j_H}} = \left( \sigma_{1,j_1} \, R_{1,j_1} \right) \dots \left( \sigma_{H-1,j_{H-1}} \, R_{H-1,j_{H-1}} \right) \left( R_{H,j_H} \right) = \nonumber \\
&=& \frac{t_{j_1,...,j_{H-1},j_{H_{k_p}}}}{\sigma_{H,j_{H_{k_p}}}} = \left( \sigma_{1,j_1} \, R_{1,j_1} \right) \dots \left( \sigma_{H-1,j_{H-1}} \, R_{H-1,j_{H-1}} \right) \left( R_{H,j_{H_{k_p}}} \right)
\end{eqnarray}

Notice that, for the appropriate multiindexes $I=(i_1,...,i_H)$ and  $J=(j_1,...,j_H)$, the disks $\widetilde{B^H_{k_p}}$ are precisely the disks $D^{I}_{J}$ from \eqref{FormulaDIJAndDIJPrime}. In particular, the disks $\widetilde{B^H_{k_p}}$ result from applying a dilation of ratio $r(B^{H-1}_{i_0}) = t_{j_1,...,j_{H-1}}= \left( \sigma_{1,j_1} \, R_{1,j_1} \right) \dots \left( \sigma_{H-1,j_{H-1}} \, R_{H-1,j_{H-1}} \right)$ (and an appropriate translation) to the disks chosen in the $N^{th}$ step, as in \eqref{AreaCoveredInNthStep}, and they are all contained in $B^{H-1}_{i_0}$.

\begin{figure}[ht]
\begin{center}
\includegraphics{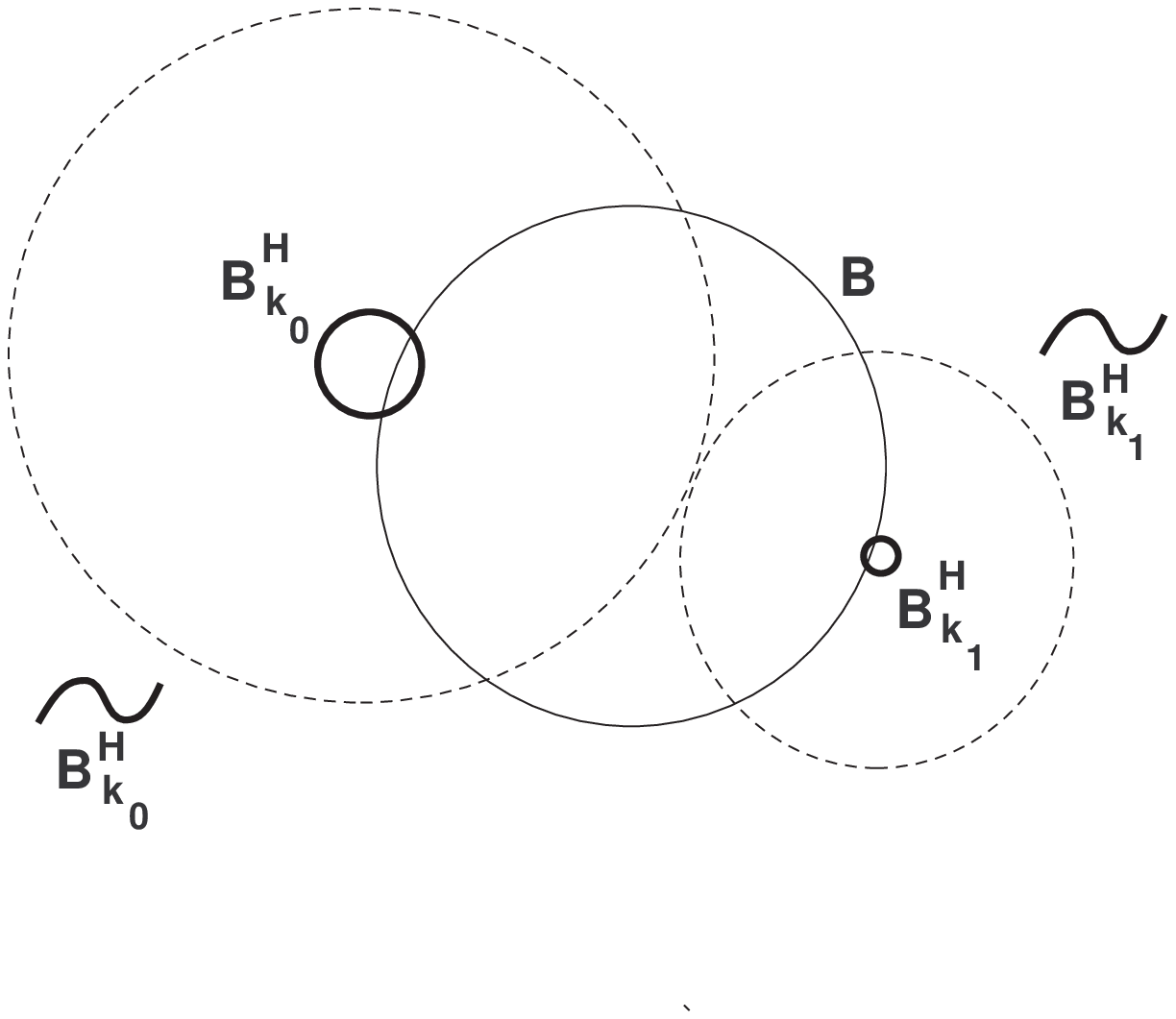}
\end{center}
\end{figure}

Recall now (see \eqref{SigmaAsFunctionOfR}) that the parameters $R_{N,p}$ are chosen so small that the parameters $\sigma_{N,p}$ are also quite small, say $< \, \frac{1}{100}$. Given that $B \cap B^H_{k_p} \neq \emptyset $ for $p=0,1, \dots , m $, ($m \geq 1$) this implies 
\begin{equation}\label{CompareRadiusBWithRadiusWideTildeBHkp}
2 \, r(B) \geq \frac{99}{100} \,\, r(\widetilde{B^H_{k_p}})
\end{equation}
for $p=0,1, \dots , m $, since $B^H_{k_p}$ is a disk concentric to $\widetilde{B^H_{k_p}}$, tiny in comparison with $\widetilde{B^H_{k_p}}$, and the disks $\widetilde{B^H_{k_p}}$ are pairwise disjoint.

Consequently, for $p=0,1, \dots , m $,
\begin{equation}\label{DilatedBCoversBHkp}
B^H_{k_p} \subset 2B \,\,\,\,\, \text{ and } \,\,\,\,\, \widetilde{B^H_{k_p}} \subset 4B.
\end{equation}

And now, by \eqref{SummingInAllChildrenOfD}, \eqref{RadiusBHkpAsFunctionOfP}, \eqref{RadiusTargetNthStep}, \eqref{SimplificationOfTargetEquationToArea}, and \eqref{RadiusWideTildeBHkp},
\begin{eqnarray}\label{ProofPackingTarget}
\sum_{B_n \in \mathcal{C} } r(B_n)^{d'}  & \leq & \sum_{p=0}^{m} r(B^H_{k_p})^{d'} = \nonumber \\
& = & \left[   \left( \sigma_{1,\,j_1} \, R_{1,\,j_1} \right) \dots \left( \sigma_{H-1,\,j_{H-1}} \, R_{H-1,\,j_{H-1}} \right)   \right]^{d'} \, \sum_{p=0}^{m} \left( \sigma_{H,\,j_{H_{k_p}}} \, R_{H,\,j_{H_{k_p}}} \right)^{d'} = \nonumber \\
& = & \left[  t_{j_1,...,j_{H-1}}  \right]^{d'} \, \sum_{p=0}^{m} \left( R_{H,\,j_{H_{k_p}}} \right)^{2} = \nonumber \\
& = & \left[  t_{j_1,...,j_{H-1}}  \right]^{d'} \, \frac{1}{\pi} \, \sum_{p=0}^{m}  area  \left(  D (z_{H,j_{H_{k_p}}}^{k_p},R_{H,j_{H_{k_p}}} )  \right). 
\end{eqnarray}

The notation for the last step of \eqref{ProofPackingTarget} is the following. The disks  $D(z_{H,j_{H_{k_p}}}^{k_p},R_{H,j_{H_{k_p}}})$ are those chosen in the induction step in section \ref{BasicConstruction} (see  \eqref{AreaCoveredInNthStep} and the paragraph before it.)
Since $B^{H-1}_{i_0} = \psi^{i_1}_{1,j_1}  \circ \dots \circ \psi^{i_{H-1}}_{H-1,j_{H-1}} (\overline{\D}) = f^{(H-1)}_{i_0} ( \, \overline{\D} \, )$ for some choice of $i_1, \dots ,i_{H-1}$ and of $j_1, \dots ,j_{H-1}$, then $f^{(H-1)}_{i_0} ( \, D (z_{H,j_{H_{k_p}}}^{k_p},R_{H,j_{H_{k_p}}} )  \, ) =  \widetilde{B^H_{k_p}} $. With respect to the radii, $ r\left(  D(z_{H,j_{H_{k_p}}}^{k_p},R_{H,j_{H_{k_p}}})  \right)  = R_{H,j_{H_{k_p}}} = \frac{1}{t_{j_1,...,j_{H-1}}} \, r\left(  \widetilde{B^H_{k_p}} \right)$.
Now using the fact that the disks $D(z_{H,j_{H_{k_p}}}^{k_p},R_{H,j_{H_{k_p}}})$ are pairwise disjoint, \eqref{DilatedBCoversBHkp} and the observation that $f^{(H-1)}_{i_0}$ is a dilation of ratio $t_{j_1,...,j_{H-1}}$ composed with a translation, we get that
\begin{eqnarray}\label{ProofPackingTargetPart2}
\left[  t_{j_1,...,j_{H-1}}  \right]^{d'}  &  \frac{1}{\pi} &  \, \sum_{p=0}^{m}  area  \left(  D (z_{H,j_{H_{k_p}}}^{k_p},R_{H,j_{H_{k_p}}} )  \right) \leq  
%
 \left[  t_{j_1,...,j_{H-1}}  \right]^{d'} \, \frac{1}{\pi} \,  area  \left( \left(f^{(H-1)}_{i_0}\right)^{-1} (4B) \right) =  \nonumber \\
& = \, & \left[  t_{j_1,...,j_{H-1}}  \right]^{d'} \, \left[ r \, ( \left(f^{(H-1)}_{i_0}\right)^{-1} (4B)  \, ) \right]^2 =  
%
 \left[  t_{j_1,...,j_{H-1}}  \right]^{d'} \, \left( \frac{r(4B)}{ t_{j_1,...,j_{H-1}} }   \right)^2 \lesssim   \nonumber \\
& \lesssim \, & \left[  t_{j_1,...,j_{H-1}}  \right]^{d'} \, \left( \frac{r(4B)}{ t_{j_1,...,j_{H-1}} }   \right)^{d'} \lesssim r(B)^{d'}
\end{eqnarray}
where, in the first inequality, we used \eqref{wlogRBLessThanRBHMinus1} and that $0<d'<2$.

Putting together \eqref{ProofPackingTarget} and \eqref{ProofPackingTargetPart2}, gives Lemma \ref{LemmaPackingTarget}.

\end{proof}

Now take a finite covering $\{U_j\}$ of $\phi(E)$ by open disks of diameter $\diam(U_j)\leq\delta$ and let $\delta_0 >0$ be the Lebesgue number of $\{ U_j \}$ (i.e. if $A \subseteq \phi(E)$ with $\diam(A) \leq \delta_0$, then $A \subset U_i$ for some $i$, see e.g. \cite{willard} p.163.)
Denote by $N_0$ the minimal integer such that $t_{j_1,...,j_{N_0}}\leq\delta_0$, for all possible choices of $j_1,...,j_{N_0}$ (recall \eqref{RadiusTargetNthStep}, and that $\sigma_{k,j_k} < \frac{1}{100}$ and $R_{k,j_k} < \frac{1}{2}$ for all $k$.)

By construction, the family  $\ds \left\{ \psi^{i_1}_{1,j_1}  \circ \dots \circ \psi^{i_{N_0}}_{N_0,j_{N_0}} \left( \, \overline{\D} \, \right) \right\}_{\substack{i_1,\dots,i_{N_0}\\j_1,\dots,j_{N_0} } }$ is a covering of $\phi(E)$ with the $\H^{d'}$-packing condition \cite{mattila} (i.e. satisfying Lemma \ref{LemmaPackingTarget}.) Thus, 
$$\aligned
\sum_j \left[\diam(U_j)\right]^{d'}&\geq C\,\sum_{\substack{i_1,\dots,i_{N_0}\\j_1,\dots,j_{N_0} } }
\left[ \psi^{i_1}_{1,j_1}  \circ \dots \circ \psi^{i_{N_0}}_{N_0,j_{N_0}} \left( \, \overline{\D} \, \right)  \right]^{d'} \geq C'= C \prod_{n=1}^{\infty} \left( 1- \varepsilon_n \right) >0
\endaligned$$

Hence, $\H^{d'}_{\delta}(\phi(E))\geq C'$ and letting $\delta\to 0$, we get that
$$0 < C' \leq\H^{d'}(\phi(E)) .$$ 

A similar argument, based this time on \eqref{GeneralEquationSourceSetAtScaleN}, (and hence substituting $d$ for $d'$ and $\sigma^K$ for $\sigma$, etc.) gives that $1 \lesssim \H^{d}(E)$.

The positive answer to Question \ref{question4.2inACMOU} is now readily obtained from the aforementioned result of Kaufman \cite{kaufman} and Kr\'{a}l \cite{kral}, that the condition $\H^1(E)=0$ is a precise characterization for removable singularities of  $BMO$ analytic functions. 

This finishes the proof of Theorem \ref{TheoremConstructionForUsualHausdorffMeasures}.
\end{proof}

From the above proof we want to remark that a key idea of it is to choose appropriately the parameters $\sigma$, as in \eqref{SimplificationOfTargetEquationToArea}, so that the expressions appearing in the calculation of the Hausdorff measure of the set in question (see \eqref{GeneralEquationTargetSetAtScaleN}), actually can be translated into area, which is indeed additive. This geometric idea is the proposed way of handling Cantor-type sets with building blocks at the same step of very different sizes: that the distribution of the building blocks be uniform with respect to area in the same way it was done in the above calculations. 

More precisely, the philosophy (for this heuristic comment let us work with squares, but the same idea works for circles as above) is that if one wants to build a Cantor set inside a square $\widetilde{Q}$ of sidelength $1$, and one wants to have (say) $9$ children $S_1, ... , S_9$ of equal size, then one can divide $\widetilde{Q}$ into $9$ equal squares $Q_1, ... , Q_9$ and place inside each $Q_i$ a square $S_i$ of sidelength $l(S_i)$, with the same center as $Q_i$, which satisfies that $l(S_i)^{d} = area (Q_i)$. After iteration, this construction yields a Cantor set of strictly positive and finite Hausdorff $\H^d$ measure, as in our proof. We thank K.Astala, A. Clop, J. Mateu and J. Orobitg for insightful conversations regarding this aspect just mentioned in the case of children of equal size when we were proving Theorem 5.1 (stated here as Theorem \ref{theorem5.1inACMOU}) in our joint paper \cite{astalaclopmateuorobitguriartepreprint}. Further pondering of those conversations led us later to a deeper understanding of the philosophy explained now.

However if we want (as in our case) the square $\widetilde{Q}$ to have (say) $9$ square children $S_1, ... , S_9$ of unequal size, one can then divide $\widetilde{Q}$ into $9$ unequal squares $Q_1, ... , Q_9$ and place inside each $Q_i$ a square $S_i$ of sidelength $l(S_i)$, with the same center as $Q_i$, which satisfies that $l(S_i)^{d} = area (Q_i)$. Again after iteration, this construction yields a Cantor set of strictly positive and finite Hausdorff $\H^d$ measure.

This geometric meaning of the choice of parameters is also suggested by an algebraic insight (see \eqref{SigmaAsFunctionOfR} and \eqref{EquationSigmaRInCaseTargetDimensionIs1}.)

From Theorem \ref{TheoremConstructionForUsualHausdorffMeasures} we also obtain the following
\begin{cor}\label{Theorem1.1inACMOUSharpInSenseOfRelaxingHypothesis}
Theorem \ref{theorem1.1inACMOU} is sharp in the sense that no relaxation of the hypothesis (in terms of Hausdorff gauge functions) allows the same conclusion to hold, i.e. there exists a compact set $E \subset \C$ and a $K$-quasiconformal mapping $\phi$ such that $\H^\frac{2}{K+1}(E)>0$ and $\H^{1}(\phi(E))>0$.
\end{cor}

and, analogously
\begin{cor}\label{AbsoluteContinuityOfHausdorffMeasuresSharpInSenseOfRelaxingHypothesis}
If the implication \eqref{abscont} is true, it is sharp in the same sense as Corollary \ref{Theorem1.1inACMOUSharpInSenseOfRelaxingHypothesis}.
\end{cor}

Also, we get 
\begin{cor}\label{Theorem2.5inACMOUSharpInSenseOfStrengtheningConclusion}
Theorem \ref{theorem2.5inACMOU} is sharp (in the sigma-finite measure goes to sigma-finite measure version) in the sense that under the same hypothesis the conclusion cannot be strengthened. I.e. there exists a compact set $F \subset \C$ and a $K$-quasiconformal mapping $\phi$ such that $\H^\frac{2}{K+1}(F) = \infty$ but $F$ is $\H^\frac{2}{K+1}$-$\sigma$-finite and $\H^{1}(\phi(F))=\infty$ but $\phi(F)$ is $\H^1$-$\sigma$-finite. In general, for any $0<d<2$ there exists a compact set $F \subset \C$ and a $K$-quasiconformal mapping $\phi$ such that $\H^d(F) = \infty$ but is $\H^d$-$\sigma$-finite, and $\H^{d'}(\phi(F))=\infty$ but is $\H^{d'}$-$\sigma$-finite. As usual, we are taking $d'=\frac{2Kd}{2+(K-1)d}$.
\end{cor}

\begin{proof}
The proof of this corollary is achieved with the usual technique of ``gluing" together countably many copies of the example constructed in Theorem \ref{TheoremConstructionForUsualHausdorffMeasures}. 

More precisely, let $Q_1$ be a square of sidelength $\frac{1}{2}$, with sides parallel to the coordinate axes, and the lower side being $ [0,\frac{1}{2}] \times \{ 0 \}$.  Then let $Q_2$ be a square of sidelength $\frac{1}{4}$, with the lower side being $ [\frac{1}{2}, \frac{3}{4}] \times \{ 0 \}$, and in general let $Q_n$ be a square of sidelength $\frac{1}{2^{-n}}$, with the lower side being $ [1-\frac{1}{2^{-n-1}}, 1-\frac{1}{2^{-n}}] \times \{ 0 \}$. We will place several rescaled copies of the compact set from Theorem \ref{TheoremConstructionForUsualHausdorffMeasures} (let us call it $E \subset \D$) inside each $Q_n$ so that each $Q_n$ contributes at least $c > 0$ towards both $\H^d(E)$ and $\H^{d'}(\phi (E))$, where $\phi$ is the $K$-quasiconformal mapping defined as the identity outside the copies of $E$ that we will prescribe, and the rescaled copy of the $K$-quasiconformal mapping from Theorem \ref{TheoremConstructionForUsualHausdorffMeasures} on each copy of $E$.

Let $\lambda_n = \frac{1}{\sqrt{n} \; \log (n+2)}$. Notice that $\displaystyle{ \sum_{n=1}^{\infty} (\lambda_n)^2 < \infty }$, but $\displaystyle{ \sum_{n=1}^{\infty} (\lambda_n)^\alpha = \infty }$ for all $\alpha<2$. Fix now $d$ and $d'$ as in the statement. Let us fix $k \in \mathbb{N}$. Notice that the area of each $Q_k$ satisfies $\mid Q_k \mid = \frac{1}{4^k}$. Let $\varepsilon_k >0$ be so small that $\displaystyle{ \sum_{n=1}^{\infty} (\varepsilon_k  \lambda_n)^2 < \frac{1}{1000} \frac{1}{4^k} }$, and let $N_k$ be large enough that $\displaystyle{ \sum_{n=1}^{N_k} (\varepsilon_k  \lambda_n)^{d'} > 1 }$.

Let $c_{k,n} \in \mathbb{N}$ be such that $\varepsilon_k  \lambda_n < 2^{-c_{k,n}}  \leq 2 \varepsilon_k  \lambda_n$ (some values of $c_{k,n}$ might be repeated for different values of $k$ and $n$.) Then $\displaystyle{ \sum_{n=1}^{N_k} ( 2^{-c_{k,n}}  )^2 \leq \frac{1}{250} \frac{1}{4^k} = \frac{1}{250} \mid Q_k \mid  }$. Inside each $Q_k$ place $N_k$ dyadic squares $\{ S_{k,n} \}_{n=1}^{N_k}$ with disjoint interiors, of sidelengths (respectively) $2^{-c_{k,n}}$. E.g. subdivide $Q_k$ into its 4 dyadic children, and each of these into its 4 dyadic children until $2^{-c_{k,1}}$ is reached. Then keep as many squares of sidelength $2^{-c_{k,1}}$ as needed (to account for repetitions in the values of $c_{k,n} $), and with the remaining squares, restart the process of subdivision. The process is finite and there is no overlapping, since all squares involved are dyadic and the sum of the areas of the squares $S_{k,n}$ is smaller than the area of $Q_k$.

With the same center as $S_{k,n}$, draw $\frac{\varepsilon_k  \lambda_n}{2} E \subset \frac{\varepsilon_k  \lambda_n}{2} \D$ i.e. $E$ and $\D$ rescaled by a factor of $\frac{\varepsilon_k  \lambda_n}{2}$, so that they fit inside $S_{k,n}$. Let $F$ be the union of the countably many rescaled copies of $E$ just described together with the point $(1,0)$ and $\phi$ the $K$-quasiconformal mapping previously described (it is easy to see that it is $K$-qc, e.g. by taking the mappings $\phi_N$ that agree with (a rescaled copy of) the mapping from Theorem \ref{TheoremConstructionForUsualHausdorffMeasures} on the first $N$ copies of $E$ and are otherwise the identity, and then observing that there exists a limit $K$-quasiconformal mapping  $\ds \phi=\lim_{N\to\infty}\phi_N$ with convergence in $W^{1,p}_{loc}(\C)$ for any $p<\frac{2K}{K-1}$.) 

The set $F$ is compact, and each square $Q_k$ contributes at least $c_0 >0$ towards $\H^{d'} (\phi(F))$. Since $0 < \lambda_n  < 1$, then $(\lambda_n)^{d'} < (\lambda_n)^{d}$, so (in the source plane) each square $Q_k$ contributes at least $c_1 >0$ towards $\H^d (F)$.

\end{proof}


\begin{remark}\label{RemarkOnConsequenceForBesselCapacities}
Let us also mention that, aiming at proving implication \eqref{abscont}, in \cite{astalaclopmateuorobitguriartepreprint} a partial result is obtained for $1<t'<2$, namely that for a compact set $E \subset \C$ and a $K$-quasiconformal mapping $\phi: \C \to \C$, if $\H^{t}(E)=0$ (or even $\H^{t}(E)< \infty$), then $C_{\alpha, t'} (\phi (E)) =0$, where $\alpha = \frac{2}{t'}-1$, and $C_{\alpha,p}$ stands for the Bessel capacity. This implies then that $\H^{h} (\phi (E)) =0$ for any gauge function $h(s)=s^{t'}\,\varepsilon(s)$ such that 

\begin{equation}\label{ConditionForHausdorffGaugeFunctionsRelatedBesselCapacities}
\int_0 \varepsilon(s)^\frac{1}{t'-1}\frac{ds}{s}<\infty
\end{equation}

This might induce to wonder whether the ``correct" implication for $K$-quasiconformal mappings $\phi$ is not quite the implication \eqref{abscont}, but some sort of implication that a certain Bessel capacity is zero for $E$ implies that another Bessel capacity is zero for $\phi (E)$. In the limiting case $t'=1$, this would imply that any compact set $E$ with $0<\H^\frac{2}{K+1}(E)<\infty$ (and hence $C_{\alpha,p}(E)=0$, for some $\alpha,p$, with $p>1$), would satisfy $\H^{1}(\phi(E))=0$ (since $C_{1,1}(E) \approx \H^1_{\infty} (E)=0$, see \cite{adamshedberg}), which we just showed is false (Corollary \ref{Theorem1.1inACMOUSharpInSenseOfRelaxingHypothesis}.)
\end{remark}

\section{Examples of extremal distortion for generalized Hausdorff measures}\label{ConstructionForGeneralizedHausdorffMeasures}

Let us first agree on some terminology and notation. A {\it{measure function}} (or {\it{gauge function}}) is a continuous non-decreasing function $h(t)$, $t\geq 0$, such that $\ds \lim_{t\to 0}h(t)=0$ and $h(t)>0$ for $t>0$. If $h$ is a measure function and $F\subset\C$ the $h$-Hausdorff content of $F$ is 
$$\H^h_{\infty}(F)=\inf \sum_j h(2 r_j)$$
where the infimum is taken over all countable coverings of $F$ by disks $B(z_j, r_j)$ of diameter $d_j = 2 r_j$. If the infimum is taken over all countable coverings of $F$ by disks of radius $r_j$ with the additional restriction that $d_j < \delta$, then such infimum is denoted by $\H^h_{\delta}(F)$. Taking the limit as $\delta \to 0$, one gets the (generalized) $h$-Hausdorff measure of $F$, denoted by $\H^h (F)$. When $h(t)=t^\alpha$, one gets the (usual) $\alpha$-dimensional Hausdorff measure $\H^\alpha (F)$.

Recall that one can introduce a partial order into the family of gauge functions (see e.g. \cite{rogershausdorffmeasures}), denoted 
\begin{equation}\label{PartialOrderForGaugeFunctions1}
g \,\, \prec \,\, h
\end{equation}
whenever 
\begin{equation}\label{PartialOrderForGaugeFunctions2}
\frac{h(t)}{g(t)} \to 0 \,\, as \,\, t \to 0
\end{equation}
and it is said that $g$ corresponds to a smaller generalized dimension than $h$.

Recall also a standard comparison theorem, namely that if $f \prec g \prec h$ and the set $E$ has $\H^{g}(E)$ strictly positive and $\sigma$-finite, then $\H^{h}(E)=0$ and $\H^{f}(E)$ is non-$\sigma$-finite.

Two gauge functions $g$ and $h$ satisfying 
$$
0 < \liminf_{t \to 0} \frac{h(t)}{g(t)} \leq \limsup_{t \to 0} \frac{h(t)}{g(t)} < \infty
$$
will be regarded as (essentially) equivalent, although they would not lead to the same Hausdorff measures.

Then, for some of the generalized Hausdorff measures ``near" the usual Hausdorff measures we have the following Theorem that will allow us to prove the sharpness of Theorems \ref{theorem1.1inACMOU} and \ref{theorem2.5inACMOU} in a sense different to the one we have already proved.

\begin{thm}\label{SharpExampleTheoremForGeneralizedHausdorffMeasures}
Let $h^{(S)} (t) = t^d \varepsilon(t)$ be a gauge function where $0<d<2$, and one of the following two conditions is satisfied (for $t < t_0$):
\begin{enumerate}
\item[(a)] $\varepsilon(t)$ is a (strictly) decreasing function, $\varepsilon(t) \to \infty$ as $t \to 0$, but for all $\alpha >0$, we have that $t^\alpha \varepsilon(t) \to 0$ as $t \to 0$. 
\item[(b)] $\varepsilon(t)$ is a (strictly) increasing function, $\varepsilon(t) \to 0$ as $t \to 0$, but for all $\alpha >0$, we have that $\frac{t^\alpha}{\varepsilon(t)} \to 0$ as $t \to 0$. In this case it follows that for all $\alpha >0$, $t^\alpha \varepsilon(t)$ is a (strictly) increasing function of $t$. Let us also assume that $\frac{\varepsilon^{\frac{1}{2-d}}(t)}{t}$ is a decreasing function of $t$ (which happens e.g. if $\varepsilon^{\frac{1}{2-d}}(t)$ is concave), and the logarithmic-type condition that $\varepsilon (t) \lesssim \varepsilon (t^K)$. 
\end{enumerate}

Then, if $d'=\frac{2Kd}{2+(K-1)d}$ as in equation \eqref{abscont}, there is a compact set $E \subset \C$ and a $K$-quasiconformal mapping $\phi :\C \to \C$ such that $\H^{ h^{(S)} }(E) \approx 1$ and $\H^{\widetilde {h^{(T)} } } (\phi (E)) >0$, where, correspondingly we can take
\begin{enumerate}
\item[(a)] $\widetilde{h^{(T)} } (t)= t^{d'} \varepsilon^{\frac{2}{2+(K-1)d}} (t^K)$
\item[(b)] $\widetilde{h^{(T)} } (t)= t^{d'} \varepsilon^{\frac{2}{2+(K-1)d}} (t)$
\end{enumerate}

\end{thm}

Let us comment briefly on the notation. The superscripts $(S)$ or $(T)$ for the gauge function stand correspondingly, for the source or the target set. The reason we use a tilde on top of the target gauge function is that the actual gauge function that we get for the target $h^{(T)}$ is more complicated than $\widetilde{h^{(T)} }$, but is frequently equivalent to $\widetilde{h^{(T)} }$ and, as we will see, is always related to $\widetilde{h^{(T)} }$ via an appropriate inequality. 

The case $\varepsilon (t) =1$ has already been dealt with in Theorem \ref{TheoremConstructionForUsualHausdorffMeasures}. A typical example for case (a) in Theorem \ref{SharpExampleTheoremForGeneralizedHausdorffMeasures} is $\varepsilon (t) = \log^{\beta} \left( \frac{1}{t} \right)$, with $\beta >0$. And for case (b), a typical example is $\varepsilon (t) = \frac{1}{\log^{\beta} \left( \frac{1}{t} \right)}$, with $\beta >0$. The requirement that the gauge functions be strictly monotone (as opposed to just monotone) causes no loss of generality for the purpose of proving that the Hausdorff measure of a set is zero, positive or finite, as is well-known, so we will assume henceforth that all gauge functions are strictly monotone.

\begin{proof} (Of Theorem \ref{SharpExampleTheoremForGeneralizedHausdorffMeasures}.)
The proof of Theorem \ref{SharpExampleTheoremForGeneralizedHausdorffMeasures} is similar to that of Theorem \ref{TheoremConstructionForUsualHausdorffMeasures}, but more technical, so we will just indicate the changes needed, leaving the details for the reader.

First of all, notice that in case $(a)$, the hypothesis that $h^{(S)} (t) = t^d \varepsilon(t)$ be a gauge function implies that it is positive and non-decreasing (for $t<t_0$.) Since $t \to t^\alpha$ for $\alpha >0$ is positive and increasing for $t \geq 0$, then $\widetilde{h^{(T)} } = t^{d'} \varepsilon^{\frac{2}{2+(K-1)d}} (t^K)$ is also non-decreasing (for $t<t_0$.) In case $(b)$, it is obvious that $\widetilde{h^{(T)} }$ is non-decreasing.

We repeat the basic construction as in section \ref{BasicConstruction}. The equation corresponding to \eqref{GeneralEquationSourceSetAtScaleN} now reads
\begin{equation}\label{GeneralEquationSourceSetAtScaleNGeneralizedHausdorffMeasure}
\sum_{j_1,...,j_N}  m_{1,j_1} m_{2,j_2} \ldots m_{N,j_N} h^{(S)}( s_{j_1,...,j_N}) \approx 1
\end{equation}

and in parallel to equation \eqref{GeneralEquationTargetSetAtScaleN} we have
\begin{equation}\label{GeneralEquationTargetSetAtScaleNGeneralizedHausdorffMeasure}
\sum_{j_1,...,j_N}  m_{1,j_1} m_{2,j_2} \ldots m_{N,j_N} h^{(T)}( t_{j_1,...,j_N}) \approx 1
\end{equation}

for a certain gauge function $h^{(T)}( t)$ to be determined later.

In analogy to equation \eqref{SigmaAsFunctionOfR}, we inductively define the parameters $\sigma_{k,j_k}$ by
\begin{equation}\label{SigmaAsFunctionOfRGeneralizedHausdorffMeasure}
\left[ (\sigma_{1,j_1})^K \, R_{1,j_1} \dots  (\sigma_{N,j_N})^K \, R_{N,j_N} \right]^d \, \varepsilon \left\{ (\sigma_{1,j_1})^K \, R_{1,j_1} \dots  (\sigma_{N,j_N})^K \, R_{N,j_N}   \right\} = \left(  R_{1,j_1} \dots R_{N,j_N} \right)^2
\end{equation}

One technicality appearing now is that, in general, the $\varepsilon (t)$ terms are not going to be multiplicative as we run from one step to the next, whereas in the case $\varepsilon (t) =1$, they are multiplicative. The inductive definition \eqref{SigmaAsFunctionOfRGeneralizedHausdorffMeasure} can be made, since $h^{(S)}(t) = t^d \varepsilon (t)$, being a measure function, is strictly increasing for $t<t_0$, hence it is injective. We will later be more specific regarding further restrictions on the choice of the parameters, but for the time being let us notice that if need be, we can choose the radii $R_{k,j_k}$ inductively so that $R_{l,i} < R_{m,j}$ if $l>m$, for all $i,j$. The parameters $\sigma_{k,j_k}$ can be taken as small as we wish by taking the $R_{k,j_k}$ sufficiently small. One (coarse) way to verify this in case (b) of Theorem \ref{SharpExampleTheoremForGeneralizedHausdorffMeasures}, is to notice that for $t<t_0$, and for $0< \delta < 2-d$, one has $t^d > h^{(S)}(t) = t^d \varepsilon (t) > t^{d+\delta}$. Consequently, a comparison of the solutions of the equation $h^{(S)}(t) = (R_{1,j_1})^2$ with the corresponding equations where the $h^{(S)}(t)$ is replaced by $t^d$ and $t^{d+\delta}$, and an inductive argument using \eqref{SigmaAsFunctionOfRGeneralizedHausdorffMeasure}, gives
\begin{eqnarray}\label{EstimateForSigmaNFirstEquation}
 \left[ (\sigma_{1,j_1})^K \, R_{1,j_1} \right. & \dots & \left.  (\sigma_{N,j_N})^K \, R_{N,j_N} \right]^d > \nonumber \\
&>& \left[ (\sigma_{1,j_1})^K \, R_{1,j_1} \dots  (\sigma_{N,j_N})^K \, R_{N,j_N} \right]^d \, \varepsilon \left\{ (\sigma_{1,j_1})^K \, R_{1,j_1} \dots  (\sigma_{N,j_N})^K \, R_{N,j_N}   \right\} > \nonumber \\
&>& \left[ (\sigma_{1,j_1})^K \, R_{1,j_1} \dots  (\sigma_{N,j_N})^K \, R_{N,j_N} \right]^{d+\delta}
\end{eqnarray}
and
\begin{equation}\label{EstimateForSigmaNSecondEquation}
\sigma_{N,j_N} < \frac{ \left( R_{1,j_1} \dots  R_{N,j_N}  \right)^{ \frac{2-d-\delta}{K(d+\delta)} }  }{ \sigma_{1,j_1} \dots \sigma_{N-1,j_{N-1}} } 
< \left(  R_{N,j_N}  \right)^{ \frac{2-d-\delta}{K(d+\delta)} } \left( R_{1,j_1} \dots R_{N-1,j_{N-1}}  \right)^{ \frac{2-d-\delta}{K(d+\delta)} - \frac{2-d}{Kd} }
\end{equation}

which shows that, once the $R_{k,j_k}$ have been chosen for $1 \leq k \leq N-1$, then $\sigma_{N,j_N}$ can be made as small as we want by choosing $R_{N,j_N}$ sufficiently small.

A parallel reasoning for the case (a) of Theorem \ref{SharpExampleTheoremForGeneralizedHausdorffMeasures} gives
\begin{equation}\label{EstimateForSigmaNThirdEquation}
\sigma_{N,j_N} < \frac{ \left( R_{1,j_1} \dots  R_{N,j_N}  \right)^{
\frac{2-d}{Kd}} }{ \sigma_{1,j_1} \dots \sigma_{N-1,j_{N-1}} } 
< \left(  R_{N,j_N}  \right)^{ \frac{2-d}{Kd}    } \left( R_{1,j_1} \dots R_{N-1,j_{N-1}}  \right)^{ \frac{2-d}{Kd} - \frac{2-d+\delta}{K(d-\delta)}  }
\end{equation}

so that here we would choose $0< \delta <d$ to reach the same conclusion (that the parameters $\sigma_{k,j_k}$ can be taken as small as we wish by taking the $R_{k,j_k}$ sufficiently small.)


Going back to the main thread of the argument, now the same reasoning as in Theorem \ref{TheoremConstructionForUsualHausdorffMeasures}
where the role of \eqref{SimplificationOfTargetEquationToArea} is played by \eqref{SigmaAsFunctionOfRGeneralizedHausdorffMeasure}
yields as in \eqref{UpperHausdorffMeasureEstimateForSource} that $\H^{h^{(S)}} (E) \lesssim 1$.

Regarding the target set $\phi(E)$, in parallel to \eqref{SimplificationOfTargetEquationToArea}, the definition we made of the parameters $\sigma_{k,j_k}$ in \eqref{SigmaAsFunctionOfRGeneralizedHausdorffMeasure} implies (recall $d' = \frac{2Kd}{2+(K-1)d} $):
\begin{equation}\label{SigmaAsFunctionOfRGeneralizedHausdorffMeasureForTargetSet}
\left[ \sigma_{1,j_1} \, R_{1,j_1} \dots  \sigma_{N,j_N} \, R_{N,j_N} \right]^{d'} \, 
\varepsilon^{ \frac{2}{2+(K-1)d} } \left\{ (\sigma_{1,j_1})^K \, R_{1,j_1} \dots  (\sigma_{N,j_N})^K \, R_{N,j_N}   \right\} 
= \left(  R_{1,j_1} \dots R_{N,j_N} \right)^2
\end{equation}

i.e. we have (see \eqref{RadiusSourceNthStep} and \eqref{RadiusTargetNthStep})
\begin{equation}\label{DefinitionOfH(T)}
h^{(T)} (t_{j_1,...,j_N}) = (t_{j_1,...,j_N})^{d'}   \varepsilon^{ \frac{2}{2+(K-1)d} } \left( s_{j_1,...,j_N} \right) = \left(  R_{1,j_1} \dots R_{N,j_N} \right)^2
\end{equation}

which can be used as the definition of the gauge function $h^{(T)} (t)$ in a countable number of points $t_{j_1,...,j_N} \to 0$, as $N \to \infty$, with the intention of later extending it in a continuous strictly increasing fashion, provided some technical nuances are taken care of.

If this extension can be done, we readily have $\H^{h^{(T)}} (\phi(E)) \lesssim 1$, in parallel to Lemma \ref{UpperBoundsForHausdorffMeasureInSourceAndTarget}.

The first point to be taken care of is that it is not a priori clear that the correspondence $t_{j_1,...,j_N} \to s_{j_1,...,j_N}$ is actually a function, i.e. a priori a certain choice of parameters $\sigma_{k,j_k}$ and $R_{k,j_k}$ could be made such that $t_{j_1,...,j_N} = t_{i_1,...,i_M}$ but $s_{j_1,...,j_N} \neq s_{i_1,...,i_M}$, which would create a problem. One way to avoid this problem is to first choose $\widetilde{\varepsilon_1}$ as a candidate for $\varepsilon_1$, then choose $R_{1,1} < \min\{ \frac{1}{100}, t_0 \}$ and so small that $\sigma_{1,1} < \min\{ \frac{1}{100}, t_0 \}$. Now choose $R_{1,2}$ so small that both $R_{1,2}$ and $\sigma_{1,2}$ are smaller than the previously chosen parameters $R_{1,1}$ and $\sigma_{1,1}$. Proceed in this way, each time making sure that the parameters $R_{1,j}$ and $\sigma_{1,j}$ are smaller than all the previously chosen parameters $R_{1,k}$ and $\sigma_{1,k}$. Notice that the proof of Lemma \ref{FillingAreaOfDiskWithDisks} part (b) works no matter how small the radii $R_{1,k}$ are taken. Once the proportion of area of $\D$ occupied in this way 
by the disks of radius $R_{1,k}$, for $1 \leq k \leq l_1$ is larger than $(1-\widetilde{\varepsilon_1})$, stop and call that proportion of area $(1-\varepsilon_1)$. This way $0< \varepsilon_1 < \widetilde{\varepsilon_1}$.

Next choose $\widetilde{\varepsilon_2}$ as a candidate for $\varepsilon_2$. Choose now $R_{2,1}$ so small that both $R_{2,1}$ and $\sigma_{2,1}$ are smaller than all previously chosen parameters $R_{1,k}$ and $\sigma_{1,k}$. Now choose $R_{2,2}$ so small that both $R_{2,2}$ and $\sigma_{2,2}$ are smaller than the previously chosen parameters $R_{1,k}$ and $\sigma_{1,k}$ and $R_{2,1}$ and $\sigma_{2,1}$. We also require from $R_{2,2}$ and $\sigma_{2,2}$ that $R_{2,2}$ be chosen so small that
$$
(\sigma_{2,2} \; R_{2,2}) (\sigma_{1,k} \; R_{1,k} ) < \min_{l} \{ (\sigma_{2,1} \; R_{2,1}) (\sigma_{1,l} \; R_{1,l} )   \}, \text{  for all }k.
$$
Proceed in this way for the choice of parameters. I.e., choose inductively $R_{N,j_N}$ so small that both $R_{N,j_N}$ and $\sigma_{N,j_N}$ are smaller than all previously chosen parameters $R_{m,k}$ and $\sigma_{m,k}$, and also that all possible products of the form $t_{j_1,...,j_N}=\left( \sigma_{1,j_1} \, R_{1,j_1} \right) \dots \left( \sigma_{N,j_N} \, R_{N,j_N} \right)$ are smaller than all possible products of the same form that are formed with all possible combinations of parameters $R_{m,k}$ and $\sigma_{m,k}$ previously chosen. Also, since $0< \varepsilon_n < \widetilde{\varepsilon_n}$, if $\ds \prod_{n=1}^{\infty} \left( 1- \widetilde{\varepsilon_n} \right) \approx 1$, then 
$\ds \prod_{n=1}^{\infty} \left( 1- \varepsilon_n \right) \approx 1$. 

This way of choosing the parameters ensures that for any given value of $t_{j_1,...,j_N}$, there is a unique choice of $N$, $\sigma_{k,j_k}$ and $R_{k,j_k}$ such that \eqref{RadiusTargetNthStep} holds. And hence, that unique choice of $N$, $\sigma_{k,j_k}$ and $R_{k,j_k}$ yields a unique $s_{j_1,...,j_N}$ satisfying \eqref{RadiusSourceNthStep}.

Another point to take care of is that $h^{(T)} (t)$ be a strictly increasing function on the points $t_{j_1,...,j_N}$ which then allows for a continuous, strictly increasing extension to $[0, +\infty)$. In case (b) of Theorem \ref{SharpExampleTheoremForGeneralizedHausdorffMeasures}, $\varepsilon (t)$ is positive and increasing, and since the choice of parameters just described when explaining how to ensure that the correspondence $t_{j_1,...,j_N} \to s_{j_1,...,j_N}$ is actually a function can be made in the way we described so that $t_{j_1,...,j_N} \to s_{j_1,...,j_N}$ is a strictly increasing function on the set $T := \{t_{j_1,...,j_N} \}_{\substack{j_1,...,j_N \\N=1,2, \dots } }$, then $h^{(T)} (t)$ (see \eqref{DefinitionOfH(T)}) is also a strictly increasing function on $T$. In case (a) of Theorem \ref{SharpExampleTheoremForGeneralizedHausdorffMeasures}, notice that, given that all $\sigma_{k,j_k}$ and $R_{k,j_k}$ are $<1$, and $K>1$, comparing \eqref{RadiusSourceNthStep} and \eqref{RadiusTargetNthStep}, we see that for all possible choices of $N$, and $j_1,...,j_N$,
\begin{equation}\label{PowerComparisonOfSNAndTN}
\left(  t_{j_1,...,j_N}  \right)^K < s_{j_1,...,j_N} < t_{j_1,...,j_N}
\end{equation}

Note that this comparison \eqref{PowerComparisonOfSNAndTN} is what allows to conclude the expression for $\widetilde{ h^{(T)}} (t)$. Thinking of the logarithmic examples and \eqref{PowerComparisonOfSNAndTN}, the substitution of $h^{(T)} (t)$ by $\widetilde{ h^{(T)}} (t)$ is actually sharper than it might seem at first sight. (This substitution was made for convenience so that the statement of Theorem \ref{SharpExampleTheoremForGeneralizedHausdorffMeasures} looked more self-contained and to avoid explaining in the statement of the Theorem that the correspondence $ t_{j_1,...,j_N} \to  s_{j_1,...,j_N}$ is actually a function with the appropriate choice of parameters.) So, as a consequence of \eqref{PowerComparisonOfSNAndTN}, we get 
\begin{eqnarray}\label{ComparisonForH(T)AndWidetildeH(T)}
h^{(T)} (t_{j_1,...,j_N}) &=& (t_{j_1,...,j_N})^{d'}   \varepsilon^{ \frac{2}{2+(K-1)d} } \left( s_{j_1,...,j_N} \right) \ < \nonumber \\
&<& (t_{j_1,...,j_N})^{d'}   \varepsilon^{ \frac{2}{2+(K-1)d} } \left(  \left( t_{j_1,...,j_N} \right)^K   \right) = \widetilde{ h^{(T)}} (t_{j_1,...,j_N})
\end{eqnarray}

Then, to insure that $h^{(T)} (t)$ is a strictly increasing function on $S$, we can further restrict the choice of parameters previously described, and insist that $R_{1,2}$ be so small that, recalling that $t_2 = \sigma_{1,2} R_{1,2}$, then  $\widetilde{ h^{(T)}} (t_2) < h^{(T)} (t_1)$, and so on. So that in general, inductively, the right-hand side of \eqref{ComparisonForH(T)AndWidetildeH(T)} for the parameters we are choosing at a given time, is forced to be smaller than the left-hand side of \eqref{ComparisonForH(T)AndWidetildeH(T)} for all the parameters previously chosen. Alternatively, one can inductively choose $R_{k,j_k}$ to insure that $h^{(T)}$ is strictly increasing simply based on \eqref{DefinitionOfH(T)}.

As a consequence, as we mentioned above, we have that $\H^{h^{(T)}} (\phi(E)) \lesssim 1$, in parallel to Lemma \ref{UpperBoundsForHausdorffMeasureInSourceAndTarget}.

The analogous equations to \eqref{SummingInAllChildrenOfD} and \eqref{SummingInAllNGenerationDescendantsOfASubfamilyOfD} are easily seen to hold with analogous proof for $h^{(S)} (t)$ instead of $h^{d'}(t)=t^{d'}$, considering \eqref{SigmaAsFunctionOfRGeneralizedHausdorffMeasure} instead of \eqref{SimplificationOfTargetEquationToArea}, to go from an expression with $h^{(S)} (t)$ to an expression with the radii of the balls involved and back to an expression with $h^{(S)} (t)$. 

In parallel to Lemma \ref{LemmaPackingTarget}, and under the same hypotheses as Lemma \ref{LemmaPackingTarget}, the packing condition 
\begin{equation}\label{PackingSourceGeneralizedHausdorffMeasures}
\sum_{ B_n \in \mathcal{C}  } h^{(S)} \left(r(B_n) \right) \leq C_1 \, h^{(S)} \left( r(B) \right)
\end{equation}\\
for some absolute constant $C_1$ is proved in a similar way that we will now sketch.

With the same reasoning as and similar notation to Lemma \ref{LemmaPackingTarget}, we can assume that $\mathcal{C}$ is a finite family and that all the $B_n \in \mathcal{C}$ are descendants of some siblings $B^H_{k_0}, B^H_{k_1}, \dots , B^H_{k_m}$, with $m \geq 1$ (i.e. at least there are two siblings), which have a common father $B^{H-1}_{i_0}$. We assume that $\{ B^H_{k_p} \}_{p=0}^{m}$ is a complete list of ancestors of generation $H$ of the family $\{ B_n \}$. We can also similarly assume without loss of generality that, in parallel to \eqref{wlogRBLessThanRBHMinus1}, 
\begin{equation}\label{wlogRBLessThanRBHMinus1GeneralizedHausdorffMeasures}
r(B) \leq r(B^{H-1}_{i_0}) = s_{j_1,...,j_{H-1}}= \left( (\sigma_{1,j_1})^K \, R_{1,j_1} \right) \dots \left( (\sigma_{H-1,j_{H-1}})^K \, R_{H-1,j_{H-1}} \right).
\end{equation}

where, in analogy to \eqref{RadiusBHkpAsFunctionOfP} we denote
\begin{equation}\label{RadiusBHkpAsFunctionOfPGeneralizedHausdorffMeasures}
r(B^H_{k_p})= s_{j_1,...,j_{H-1},j_{H_{k_p}}}=\left( (\sigma_{1,j_1})^K \, R_{1,j_1} \right) \dots \left( (\sigma_{H-1,j_{H-1}})^K \, R_{H-1,j_{H-1}} \right) \left( (\sigma_{H,j_{H_{k_p}}})^K \, R_{H,j_{H_{k_p}}} \right).
\end{equation}

and as in \eqref{DilatedBCoversBHkp}, we have that for $p=0,1, \dots , m $, with the corresponding definition for $\widetilde{B^H_{k_p}}$ (using the parameters $\sigma^K$ instead of the parameters $\sigma$),
\begin{equation}\label{DilatedBCoversBHkpGeneralizedHausdorffMeasures}
B^H_{k_p} \subset 2B \,\,\,\,\, \text{ and } \,\,\,\,\, \widetilde{B^H_{k_p}} \subset 4B.
\end{equation}

However, we change slightly the dilation argument in \eqref{ProofPackingTarget} and \eqref{ProofPackingTargetPart2}, since the one in those equations is not best suited for the function $\varepsilon (t)$. Namely, if $\{ B_n \}$ is a disjoint finite family of building blocks contained in $B$, then as in \eqref{ProofPackingTarget} and \eqref{ProofPackingTargetPart2}, using \eqref{RadiusBHkpAsFunctionOfPGeneralizedHausdorffMeasures} and \eqref{SigmaAsFunctionOfRGeneralizedHausdorffMeasure},
\begin{eqnarray}\label{ProofPackingSourceGaneralizedHausdorffMeasures}
& \ds{\sum_{B_n \in \mathcal{C} }} & h^{(S)} \left(r(B_n) \right) \; \leq \; 
\sum_{p=0}^{m} h^{(S)} \left( r(B^H_{k_p}) \right) = \nonumber  \\
& \ds{= \sum_{p=0}^{m}} & \left[ (\sigma_{1,j_1})^K \, R_{1,j_1} \dots  
(\sigma_{H,j_{H_{k_p}}})^K \, R_{H,j_{H_{k_p}}} \right]^d \, \varepsilon \left\{ (\sigma_{1,j_1})^K \, R_{1,j_1} \dots  (\sigma_{H,j_{H_{k_p}}})^K \, R_{H,j_{H_{k_p}}}   \right\} =  \nonumber  \\
& \ds{= \sum_{p=0}^{m}} & \left(  R_{1,j_1} \dots R_{H,j_{H_{k_p}}} \right)^2 \lesssim
\left\{ \frac{r(B)}{ \left(   \sigma_{1,j_1} \; \dots \;  \sigma_{H-1,j_{H-1}} \right)^K } \right\}^2
\end{eqnarray}\\
where in the last step we used \eqref{DilatedBCoversBHkpGeneralizedHausdorffMeasures} and an appropriate argument using a dilation of ratio $\frac{1}{\left(   \sigma_{1,j_1} \; \dots \;  \sigma_{H-1,j_{H-1}} \right)^K}$ composed with a translation that we will describe momentarily. 
Since we have that $B^{H-1}_{i_0} = \varphi^{i_1}_{1,j_1}  \circ \dots \circ \varphi^{i_{H-1}}_{H-1,j_{H-1}} (\overline{\D}) = g^{(H-1)}_{i_0} ( \, \overline{\D} \, )$, for some choice of $i_1, \dots ,i_{H-1}$ and of $j_1, \dots ,j_{H-1}$, where $\varphi^{i_k}_{k,j_k}(z) = z^{i_k}_{k,j_k} + (\sigma_{k,j_k})^K \, R_{k,j_k} \, z$, if we denote by $\widetilde{ \varphi^{i_k}_{k,j_k}(z) } = z^{i_k}_{k,j_k} + R_{k,j_k} \, z$, and $\widetilde{ g^{(H-1)}_{i_0} } = \widetilde{ \varphi^{i_1}_{1,j_1} } \circ \dots \circ \widetilde{ \varphi^{i_{H-1}}_{H-1,j_{H-1}} }$, then $\pi \left(  R_{1,j_1} \dots R_{H,j_{H_{k_p}}} \right)^2$ is the area of $\left[ \widetilde{ g^{(H-1)}_{i_0} } \circ \left( g^{(H-1)}_{i_0}  \right)^{-1}  \right]    \left( \widetilde{B^H_{k_p}} \right)$, and 
the dilation argument follows by \eqref{DilatedBCoversBHkpGeneralizedHausdorffMeasures}, since the area of $\left[ \widetilde{ g^{(H-1)}_{i_0} } \circ \left( g^{(H-1)}_{i_0}  \right)^{-1}  \right]  \left( 4B \right)$ is $16 \pi \left\{ \frac{r(B)}{ \left(   \sigma_{1,j_1} \; \dots \;  \sigma_{H-1,j_{H-1}} \right)^K } \right\}^2$. I.e. the dilation composed with the translation $ \left[ \widetilde{ g^{(H-1)}_{i_0} } \circ \left( g^{(H-1)}_{i_0}  \right)^{-1}  \right] $ corresponds to repeating the basic construction with all the parameters $\sigma$ taken to be $=1$.

In order to complete the proof of the packing condition \eqref{PackingSourceGeneralizedHausdorffMeasures}, we will show that 
\begin{equation}\label{ProofPackingSourceGaneralizedHausdorffMeasures2}
\left\{ \frac{r(B)}{ \left(   \sigma_{1,j_1} \; \dots \;  \sigma_{H-1,j_{H-1}} \right)^K } \right\}^2 \lesssim h^{(S)} \left( r(B) \right)
\end{equation}

By \eqref{wlogRBLessThanRBHMinus1GeneralizedHausdorffMeasures}, we can write 
\begin{equation}\label{RadiusOfBInTermsOfThetaSourceGeneralizedHausdorffMeasures}
r(B)=  (\sigma_{1,j_1})^K \, R_{1,j_1} \; \dots \; (\sigma_{H-1,j_{H-1}})^K \, R_{H-1,j_{H-1}} \; \; \theta_s = s_{j_1, \dots ,j_{H-1}} \; \; \theta_s
\end{equation}
where $\ds{\max_{0 \leq p \leq m} } \left\{ (\sigma_{H,j_{H_{k_p}}})^K \, R_{H,j_{H_{k_p}}} \right\}   \leq \theta_s \leq 1$, by the equation analogous to \eqref{CompareRadiusBWithRadiusWideTildeBHkp}, and the subindex $s$ in $\theta_s$ corresponds to ``source". Using the definition that $h^{(S)} (t) = t^d \varepsilon(t)$, we see that \eqref{ProofPackingSourceGaneralizedHausdorffMeasures2} holds if and only if
\begin{equation}\label{ConditionThetaS}
\frac{ \left(   R_{1,j_1} \; \dots \;  R_{H-1,j_{H-1}} \right)^{2-d} }{ \left(   \sigma_{1,j_1} \; \dots \;  \sigma_{H-1,j_{H-1}} \right)^{Kd} } (\theta_s)^{2-d} \lesssim \varepsilon (r(B))
\end{equation}

which in turn, by \eqref{SigmaAsFunctionOfRGeneralizedHausdorffMeasure} and \eqref{RadiusOfBInTermsOfThetaSourceGeneralizedHausdorffMeasures} holds if and only if
\begin{equation}\label{FinalConditionForSourceInTermsOfTheta}
\theta_s \; \varepsilon^{\frac{1}{2-d}} (s_{j_1, \dots ,j_{H-1}} ) \; \lesssim \; \varepsilon^{\frac{1}{2-d}} (\theta_s \; s_{j_1, \dots ,j_{H-1}} )
\end{equation}

A parallel argument for the target set yields the corresponding equations to \eqref{SummingInAllChildrenOfD} and \eqref{SummingInAllNGenerationDescendantsOfASubfamilyOfD} without difficulty (since the only radii involved are those of the building blocks, i.e. those of the form $t_{j_1,...,j_N}$, for which $h^{(T)}$ is defined by \eqref{DefinitionOfH(T)}.) And also, using the parameters $\sigma$ instead of the parameters $\sigma^K$, and denoting (see \eqref{DefinitionOfH(T)})
\begin{equation}\label{DefinitionVarepsilonPrime}
\varepsilon'(t_{j_1,...,j_N})=  \varepsilon^{ \frac{2}{2+(K-1)d} } \left( s_{j_1,...,j_N} \right)
\end{equation}
then, following the proof of Lemma \ref{LemmaPackingTarget}, we can assume without loss of generality that
\begin{equation}\label{wlogRBLessThanRBHMinus1TargetGeneralizedHausdorffMeasures}
r(B) \leq r(B^{H-1}_{i_0}) = t_{j_1,...,j_{H-1}}= \left( \sigma_{1,j_1} \, R_{1,j_1} \right) \dots \left( \sigma_{H-1,j_{H-1}}\, R_{H-1,j_{H-1}} \right).
\end{equation}

Then we can write
\begin{equation}\label{RadiusOfBInTermsOfThetaTargetGeneralizedHausdorffMeasures}
r(B)=  \sigma_{1,j_1} \, R_{1,j_1} \; \dots \; \sigma_{H-1,j_{H-1}} \, R_{H-1,j_{H-1}} \; \; \theta_t = t_{j_1, \dots ,j_{H-1}} \; \; \theta_t
\end{equation}
where $\ds{\max_{0 \leq p \leq m} } \left\{ \sigma_{H,j_{H_{k_p}}} \, R_{H,j_{H_{k_p}}} \right\}   \leq \theta_t \leq 1$, by \eqref{CompareRadiusBWithRadiusWideTildeBHkp}. The subindex $t$ in $\theta_t$ corresponds to ``target".

Using \eqref{SigmaAsFunctionOfRGeneralizedHausdorffMeasureForTargetSet}, a reasoning similar to the one done previously towards \eqref{FinalConditionForSourceInTermsOfTheta} yields
\begin{equation}\label{FinalConditionForTargetInTermsOfThetaNoThetaTilde}
(\theta_t)^{ \frac{ 2(2-d) }{2+(K-1)d }  } \; \; \varepsilon^{\frac{ 2 }{2+(K-1)d } } (s_{j_1, \dots ,j_{H-1}}) \; \lesssim \; \varepsilon' (r(B))
\end{equation}
which can be rewritten as (see \eqref{DefinitionVarepsilonPrime} and \eqref{RadiusOfBInTermsOfThetaTargetGeneralizedHausdorffMeasures}) (notice the resemblance to \eqref{FinalConditionForSourceInTermsOfTheta})
\begin{equation}\label{ResultingEquationWithThetaTilde}
\theta_t \; \varepsilon^{\frac{ 1 }{2-d } }  (s_{j_1, \dots ,j_{H-1}}) \; \lesssim \;  \varepsilon^{\frac{ 1 }{2-d } }  ( \widetilde{\theta_t}  \; s_{j_1, \dots ,j_{H-1}})
\end{equation}
where $\widetilde{\theta_t}$ is defined (somewhat abusing notation) by
\begin{equation}\label{EquationDefiningThetaTilde}
\varepsilon'( t_{j_1, \dots ,j_{H-1}} \; \; \theta_t)=  \varepsilon^{ \frac{2}{2+(K-1)d} } \left( s_{j_1,...,j_{H-1}} \; \; \widetilde{\theta_t} \right)
\end{equation}
in whatever strictly increasing extension of $h^{(T)}(t)$ we choose (see \eqref{DefinitionOfH(T)} and the comment below it.) I.e. $\widetilde{\theta_t}$ is defined so that $t_{j_1,...,j_{H-1}} \theta_t \to s_{j_1,...,j_{H-1}} \widetilde{\theta_t}$ under (the continuous extension of) the map $t_{j_1,...,j_N} \to s_{j_1,...,j_N}$. We will momentarily take care of such extension. The abuse of notation comes from the fact that given that the correspondence $t_{j_1,...,j_N} \to s_{j_1,...,j_N}$ is actually a function the way we defined it, it is convenient (and that is how we do it) to define the correspondence $\theta_t \to \widetilde{\theta_t}$ piecewise and inductively in between two consecutive points of the set $T := \{t_{j_1,...,j_N} \}_{\substack{j_1,...,j_N \\N=1,2, \dots } }$. However, as it is written in \eqref{EquationDefiningThetaTilde}, given that $ \ds{\max_{0 \leq p \leq m} } \left\{ \sigma_{H,j_{H_{k_p}}} \, R_{H,j_{H_{k_p}}} \right\}  \leq \theta_t \leq 1$, and, consequently, $ \ds{\max_{0 \leq p \leq m} } \left\{ (\sigma_{H,j_{H_{k_p}}})^K \, R_{H,j_{H_{k_p}}} \right\}  \leq \widetilde{\theta_t} \leq 1$, it looks as if one should define the correspondence $\theta_t \to \widetilde{\theta_t}$ from a radius $t_{j_1,...,j_N}$ of a ball to the radii $t_{j_1,...,j_N,j_{N+1}}$ of the children of that ball, which are not in general adjacent points to $t_{j_1,...,j_N}$ in the set $T$, due to the way we chose our parameters (see the comment below \eqref{ComparisonForH(T)AndWidetildeH(T)}.)

In any case, the correspondence $\theta_t \to \widetilde{\theta_t}$ is increasing and $\theta_t =1$ corresponds to $\widetilde{\theta_t}=1$. Thinking of \eqref{PowerComparisonOfSNAndTN}, we also demand that the correspondence $\theta_t \to \widetilde{\theta_t}$ satisfies that 
\begin{equation}\label{PowerComparisonOfSNAndTNWithThetas}
\left(  t_{j_1,...,j_N} \; \theta_t \right)^K < s_{j_1,...,j_N} \; \widetilde{\theta_t} < t_{j_1,...,j_N} \; \theta_t
\end{equation}
which is easily seen to be feasible if one thinks graphically: \eqref{PowerComparisonOfSNAndTN} means that the graph of the correspondence $t_{j_1,...,j_N} \to s_{j_1,...,j_N}$ is trapped in between the graph of $t_{j_1,...,j_N} \to (t_{j_1,...,j_N})^K$ (i.e. $x \to x^K$) and $t_{j_1,...,j_N} \to t_{j_1,...,j_N}$ (i.e. $x \to x$), and \eqref{PowerComparisonOfSNAndTNWithThetas} simply requires this trapping between those graphs to continue holding in order to extend the correspondence $t_{j_1,...,j_N} \to s_{j_1,...,j_N}$ to a correspondence $t \to s$ for all $t< t_0$.
This defines an extension of $h^{(T)}$ from $T := \{t_{j_1,...,j_N} \}_{\substack{j_1,...,j_N \\N=1,2, \dots } }$ to $t< t_0$ suitable for our purposes. Namely,
\begin{equation}\label{DefinitionOfHTExtended}
h^{(T)} (t_{j_1,...,j_{H-1}} \theta_t ) = (t_{j_1,...,j_{H-1}} \theta_t )^{d'}   \varepsilon^{ \frac{2}{2+(K-1)d} } \left( s_{j_1,...,j_{H-1}} \widetilde{\theta_t}  \right) \ = (t_{j_1,...,j_{H-1}} \theta_t )^{d'} \varepsilon' (t_{j_1,...,j_{H-1}} \theta_t ).
\end{equation}
This definition of  $h^{(T)}$  and \eqref{PowerComparisonOfSNAndTNWithThetas} yield the ``appropriate" inequality relating $h^{(T)}$ and $\widetilde{h^{(T)}}$ that we mentioned right after the statement of Theorem \ref{SharpExampleTheoremForGeneralizedHausdorffMeasures}.

Going back to the source set, we readily have the proof of case $(a)$ in Theorem \ref{SharpExampleTheoremForGeneralizedHausdorffMeasures} (proving \eqref{FinalConditionForSourceInTermsOfTheta}):
\begin{equation}\label{ProofOfCaseATheoremGeneralizedHausdorffMeasuresSource}
\theta_s \; \varepsilon^{\frac{1}{2-d}} (s_{j_1, \dots ,j_{H-1}} ) \; \leq \;
\varepsilon^{\frac{1}{2-d}} (s_{j_1, \dots ,j_{H-1}} ) \; \leq \;
\varepsilon^{\frac{1}{2-d}} (\theta_s \; s_{j_1, \dots ,j_{H-1}} )
\end{equation}
since $\theta_s  \leq 1$ and $\varepsilon (t)$ is strictly decreasing. Consequently, as in the proof of Theorem \ref{TheoremConstructionForUsualHausdorffMeasures}, $\H^{ h^{(S)} } (E) >0$. Correspondingly, at the target set, for analogous reasons the proof of \eqref{ResultingEquationWithThetaTilde} is
\begin{equation}\label{ProofOfCaseATheoremGeneralizedHausdorffMeasuresTarget}
\theta_t \; \varepsilon^{\frac{1}{2-d}} (s_{j_1, \dots ,j_{H-1}} ) \; \leq \;
\varepsilon^{\frac{1}{2-d}} (s_{j_1, \dots ,j_{H-1}} ) \; \leq \;
\varepsilon^{\frac{1}{2-d}} (\widetilde{\theta_t} \; s_{j_1, \dots ,j_{H-1}} ).
\end{equation}

Consequently, $\H^{ h^{(T)} } (E) >0$ for $h^{(T)}(t)$ as we have implicitly defined it. Since $h^{(T)}(t) \leq \widetilde{h^{(T)}(t)}$ for all $t<t_0$ by \eqref{PowerComparisonOfSNAndTNWithThetas} (see also \eqref{ComparisonForH(T)AndWidetildeH(T)}), then $\H^{ \widetilde{h^{(T)}} } (E) >0$.

The proof of case $(b)$ in Theorem \ref{SharpExampleTheoremForGeneralizedHausdorffMeasures} is also quite simple now. Indeed \eqref{FinalConditionForSourceInTermsOfTheta} follows immediately from the hypothesis that $\frac{\varepsilon^{\frac{1}{2-d}}(t)}{t}$ is a decreasing function of $t$, since $\theta_s \leq 1$. As a consequence, $\H^{ h^{(S)} } (E) >0$. And \eqref{ResultingEquationWithThetaTilde} can be proven by using that $\frac{\varepsilon^{\frac{1}{2-d}}(t)}{t}$ is a decreasing function of $t$ (since $\theta_t \leq 1$), \eqref{PowerComparisonOfSNAndTN}, that $\varepsilon(t)$ is increasing, the logarithmic-type hypothesis of Theorem \ref{SharpExampleTheoremForGeneralizedHausdorffMeasures}, and \eqref{PowerComparisonOfSNAndTNWithThetas} as follows
\begin{eqnarray}\label{ProofOfCaseBTheoremGeneralizedHausdorffMeasuresTarget}
\theta_t \; \varepsilon^{\frac{1}{2-d}} (s_{j_1, \dots ,j_{H-1}} ) \; \leq \;
\varepsilon^{\frac{1}{2-d}} (\theta_t \; s_{j_1, \dots ,j_{H-1}} ) \; \leq \;
\varepsilon^{\frac{1}{2-d}} (\theta_t \; t_{j_1, \dots ,j_{H-1}} ) \; \lesssim \; \nonumber \\
\lesssim \varepsilon^{\frac{1}{2-d}} \left\{ (\theta_t \; t_{j_1, \dots ,j_{H-1}} )^K \right\} \; \leq \;
\varepsilon^{\frac{1}{2-d}} (\widetilde{\theta_t} \; s_{j_1, \dots ,j_{H-1}} ).
\end{eqnarray}

Consequently, in a similar fashion to case $(a)$, $\H^{ \widetilde{h^{(T)}} } (E) >0$. Notice however that in case $(b)$ of Theorem \ref{SharpExampleTheoremForGeneralizedHausdorffMeasures}, due to the logarithmic-type hypothesis of Theorem \ref{SharpExampleTheoremForGeneralizedHausdorffMeasures}, $\widetilde{h^{(T)}} (t) \approx h^{(T)} (t)$, (see \eqref{PowerComparisonOfSNAndTNWithThetas} and \eqref{DefinitionOfHTExtended}) so that in this case we also have $0< \H^{ \widetilde{h^{(T)}} } (E) < \infty $ (see below \eqref{DefinitionOfH(T)}.)

This finishes the proof of Theorem \ref{SharpExampleTheoremForGeneralizedHausdorffMeasures}.

\end{proof}

The next step is to show that the technical hypotheses of Theorem \ref{SharpExampleTheoremForGeneralizedHausdorffMeasures} are actually not that restrictive, in that we can always reduce to them if our purpose is to prove that Theorem \ref{theorem1.1inACMOU} has a sharp conclusion and that Theorem \ref{theorem2.5inACMOU} has a sharp hypothesis. This is essentially the content of the next two lemmata.

\begin{lemma}\label{CanAlwaysReduceToTheoremGeneralizedHausdorffMeasuresCaseA}\textbf{[Corresponding to case (a) of Theorem \ref{SharpExampleTheoremForGeneralizedHausdorffMeasures}]}
Assume that $h^{(S)} (t) = t^d \varepsilon(t)$ is a gauge function where $0<d<2$, and $\varepsilon(t) \to \infty$ as $t \to 0$, but for all $\alpha >0$, we have that $t^\alpha \varepsilon(t) \to 0$ as $t \to 0$.

Then there exists $\widetilde{\varepsilon} (t)$ satisfying the same conditions as $\varepsilon(t)$ but also $\widetilde{\varepsilon} (t)$ is strictly decreasing, and  $\widetilde{\varepsilon} (t) \leq  \varepsilon(t)$.
\end{lemma}

\begin{proof}
Take $\widetilde{\varepsilon}(t) = \inf \{ \varepsilon(s) : s \leq t \}$. Then $\widetilde{\varepsilon}(t)$ is decreasing, and $0 \leq \widetilde{\varepsilon} (t) \leq  \varepsilon(t)$, and hence for all $\alpha >0$, $t^\alpha \widetilde{\varepsilon}(t) \to 0$ as $t \to 0$. Also, $\widetilde{\varepsilon}(t)  \to \infty$,  as $t \to 0$, by the definition of limit. From the definition of $\widetilde{\varepsilon}(t)$ it also follows readily that $t^d \widetilde{\varepsilon}(t)$ is strictly increasing, since whenever $\widetilde{\varepsilon}(t)$ is locally constant, $t^d$ is strictly increasing. More precisely, notice that the $\inf$ in the definition of $\widetilde{\varepsilon}(t)$ is actually attained. Then, pick $t_0$ such that $\widetilde{\varepsilon}(t_0) = \varepsilon(t_0)$, and consider only $t < t_0$. An elementary case by case consideration according as $\widetilde{\varepsilon}(t_i) = \varepsilon(t_i)$ or $\widetilde{\varepsilon}(t_i) \neq \varepsilon(t_i)$ for $t_1 < t_2 \leq t_0$ yields 
$(t_1)^d \widetilde{\varepsilon}(t_1) <  (t_2)^d \widetilde{\varepsilon}(t_2)$ (whenever $\widetilde{\varepsilon}(t)$ is locally constant at $t_i$, compare with the largest or smallest $t^{\ast}$ such that $\widetilde{\varepsilon}(t^{\ast}) = \widetilde{\varepsilon}(t_i)$.)

In order to get $\widetilde{\varepsilon}(t)$ to be strictly decreasing, substitute the locally constant pieces of $\widetilde{\varepsilon}(t)$ by straight line segments with strictly negative slopes $m_i$ satisfying $|m_i| \leq \frac{1}{2} \frac{d}{t_0} \widetilde{\varepsilon}(t_0)$, and with $m_i$ so small that the outcome is continuous and strictly decreasing (and the value of $\widetilde{\varepsilon}(t_0)$ has at most halved.) Now smoothen the outcome (let us keep calling it $\widetilde{\varepsilon}(t)$ ) in such a way that, for $t <t_0$, $|\widetilde{\varepsilon} \: '(t)| < \frac{d}{t} \widetilde{\varepsilon}(t)$, which ensures that $\tilde{g}(t)= t^d \widetilde{\varepsilon}(t)$ is an increasing function of $t$, since $\widetilde{g} \: '(t) >0$ for $t <t_0$.
\end{proof}

%
%
%
%
%
%

And analogously,

\begin{lemma}\label{CanAlwaysReduceToTheoremGeneralizedHausdorffMeasuresCaseB}\textbf{[Corresponding to case (b) of Theorem \ref{SharpExampleTheoremForGeneralizedHausdorffMeasures}]}
Assume that $h^{(S)} (t) = t^d \varepsilon(t)$ is a gauge function where $0<d<2$, and $\varepsilon(t) \to 0$ as $t \to 0$, but for all $\alpha >0$, we have that $\frac{t^\alpha}{\varepsilon(t)} \to 0$ as $t \to 0$.

Then there exists $\widetilde{\varepsilon} (t)$ satisfying the same conditions as $\varepsilon(t)$ but also $\widetilde{\varepsilon} (t)$ is strictly increasing and satisfies that $\frac{\widetilde{\varepsilon}^{\; \frac{1}{2-d}}(t)}{t}$ is a decreasing function of $t$, and the logarithmic-type condition that $\widetilde{\varepsilon} (t) \lesssim \widetilde{\varepsilon} (t^K)$ for $t<t_0$ and $\varepsilon(t) \leq \widetilde{\varepsilon} (t)$.
\end{lemma}

\begin{proof}
We will successively modify $\varepsilon(t)$ to get each one of the desired properties. Let $\varepsilon_{1}(t) = \varepsilon^{\; \frac{1}{2-d}}(t)$. Now let $\varepsilon_2(t) = \sup \{ \varepsilon_1(s) : s \leq t \}$. If need be, modify $\varepsilon_2(t)$ slightly substituting the locally constant pieces by straight line segments with small strictly positive slope, to make $\varepsilon_2(t)$ strictly increasing. So far $\varepsilon_2(t) \geq \varepsilon_1(t)$. Hence for all $\alpha >0$, we have that $\frac{t^\alpha}{\varepsilon_2(t)} \to 0$ as $t \to 0$. Note also that $\varepsilon_2(t) \to 0$ as $t \to 0$, and that $\varepsilon_2(t)$ is strictly increasing.

Let now $\varepsilon_3(t)$ be the concave envelope of $\varepsilon_2(t)$ (i.e. $\varepsilon_2(t)$ defines the set $\{ (t,y): 0 \leq y \leq  \varepsilon_2(t) \; ; \; 0 \leq t \leq t_0 \}$, take the convex hull of that set and let the outcome define $\varepsilon_3(t)$.) Note that we are only defining these functions for $t<t_0$, and only in the very last step do they get extended to $t>0$. Then $\varepsilon_3(t) \geq \varepsilon_2(t)$, hence for all $\alpha >0$, we have that $\frac{t^\alpha}{\varepsilon_3(t)} \to 0$ as $t \to 0$. An easy argument proves by contradiction that $\varepsilon_3(t)$ is strictly increasing for $t < t_0 ' \leq t_0$. Relabel $t_0 ' $ as $t_0$. Starting from $t_0$ towards $0$, the definition of limit applied to the decreasing sequence along the $y$ axis $y_1=\varepsilon_2(t_0), y_2=2^{-M}, y_3=2^{-M-1}, \dots $ gives a corresponding sequence of values $x_1, x_2, x_3, \dots $ along the $t$ axis so that $\varepsilon_2(t) < y_i$ if $t < x_i$. These sequences determine a sequence of rectangles with vertexes $ (x_{i+1},0), (x_i,0), (x_i,y_i), (x_{i+1},y_i) .$ These rectangles lie above the graph of $\varepsilon_2(t)$. Consider for each $y_k$ the lines joining the point $(0,y_k)$ with the points $(x_{n+1},y_n), \ n=1,2, \dots , k-1$. The line (among these) with largest slope lies above $\varepsilon_3(t)$ and provides a $\delta_k$ so that $t < \delta_k \Longrightarrow \varepsilon_3(t) < y_{k-1}$. Hence  $\varepsilon_3(t) \to 0$ as $t \to 0$. Since $\varepsilon_3(t)$ is concave, then we get that $\frac{\varepsilon_3(t)}{t}$ (which is the slope of the line joining $(0,0)$ and $(t,\varepsilon_3(t))$) is a decreasing function of $t$.

Now let $t_1 = (t_0)^K$ and in general $t_n = (t_0)^{K^n} \to 0$ in a decreasing manner (we can assume that $t_0<1$), and let us define inductively the step function
\begin{equation}\label{DefinitionVarepsilon4}
\varepsilon_4(t) =
\begin{cases}
\varepsilon_3(t_0) &t\in I_0 = (t_1,t_0 ]\\
\max \{ \varepsilon_3(t_1), \frac{1}{K} \: \varepsilon_4(t_0) \} \; &t\in I_1 = (t_2,t_1 ]\\
\max \{ \varepsilon_3(t_2), \frac{1}{K} \: \varepsilon_4(t_1) \} \; &t\in I_2 = (t_3,t_2 ]\\
\vdots \\
\max \{ \varepsilon_3(t_n), \frac{1}{K} \: \varepsilon_4(t_{n-1}) \} \; &t\in I_n = (t_{n+1},t_n ]\\
\vdots \\
\end{cases}
\end{equation}

Then $\varepsilon_4(t) \geq \varepsilon_3(t)$ and hence for all $\alpha >0$, we have that $\frac{t^\alpha}{\varepsilon_4(t)} \to 0$ as $t \to 0$. Since $\varepsilon_3(t)$ is strictly increasing and $K>1$, an easy splitting into cases argument, according as $\varepsilon_4(t_k) = \varepsilon_3(t_k) $ or $\varepsilon_4(t_k) = \frac{1}{K} \: \varepsilon_4(t_{k-1})$ for $k=n-1,n$ gives that the steps of $\varepsilon_4(t)$ are strictly increasing, i.e. that $\varepsilon_4(t)$ is strictly increasing along the sequence $t_n$. 

Let us assign to $I_n$ the colour red if $\varepsilon_4(t_n) =  \frac{1}{K} \: \varepsilon_4(t_{n-1})$ and the colour black otherwise. Since $\varepsilon_4(t)$ is strictly increasing along the sequence $t_n$, and $\varepsilon_3(t) \to 0$ as $t \to 0$, if there are infinitely many black intervals $I_n$, then $\varepsilon_4(t) \to 0$ as $t \to 0$. If not, then for a sufficiently large $n_0$, $\varepsilon_4(t_{n_0 + m}) =  \frac{1}{K^m} \: \varepsilon_4(t_{n_0}) \to 0$ as $m \to \infty$. So $\varepsilon_4(t) \to 0$ as $t \to 0$.

The logarithmic-type condition $\varepsilon_4 (t) \leq K \varepsilon_4 (t^K)$ for $t<t_0$ follows readily since it holds for $t=t_n$. Let us now explicitly check that $\frac{\varepsilon_4(t)}{t}$ is decreasing along the sequence $\{ t_n \}$, since a restriction on $t_0$ will follow. Thinking of the geometric interpretation, we will call $\frac{\varepsilon_4(t)}{t}$ ``the slope at $t$ (of $\varepsilon_4(t))$" since it is the slope of the line joining $(0,0)$ and $(t,\varepsilon_4(t))$. If both $I_n$ and $I_{n+1}$ are black, then  $\frac{\varepsilon_4(t_n)}{t_n} \leq  \frac{\varepsilon_4(t_{n+1})}{t_{n+1}}$ since $\varepsilon_3(t)$ is concave. If $I_{n+1}$ is red, then the slope at $t_n$ is $\frac{\varepsilon_4(t_n)}{t_n}$ and the slope at $t_{n+1}$ is $\frac{\varepsilon_4(t_n)}{t_n} \frac{t_n}{K \: t_{n+1}}$, so we want $t_n > K \: t_{n+1} = K \: (t_n)^K$. Since $f(x)= x - Kx^K >0$ in the interval $(0, \left( \frac{1}{K} \right)^{ \frac{1}{K-1} } )$, it suffices to restrict $t_0$ to be $ t_0 < \left( \frac{1}{K} \right)^{ \frac{1}{K-1} } $. If $I_n$ is red and $I_{n+1}$ is black, then $\varepsilon_3(t_{n+1}) \geq \frac{1}{K} \varepsilon_4(t_n) $, and thus the slope at $t_n$, is bounded above as follows
$$
\frac{\varepsilon_4(t_n)}{t_n} = \frac{t_{n+1}}{t_n} \frac{\varepsilon_4(t_n)}{t_{n+1}} <
\frac{K t_{n+1}}{t_n} \frac{\varepsilon_4(t_{n+1})}{t_{n+1}} < \frac{\varepsilon_4(t_{n+1})}{t_{n+1}}
$$

again using $t_0 < \left( \frac{1}{K} \right)^{ \frac{1}{K-1} } $.

Consider now $\varepsilon_5(t)= 10 K \varepsilon_4(t)$. Since the jumps in $\varepsilon_4(t)$ from one step to the next are given by a multiplicative factor of at most $K$, $\varepsilon_5(t_{n+1}) > \varepsilon_4(t_n)$, i.e. the graphs of $\varepsilon_4(t)$ and $\varepsilon_5(t)$ leave an empty ``corridor" in between them. Of course, $\varepsilon_5(t)$ satisfies the same properties as $\varepsilon_4(t)$.

Define now $\varepsilon_6(t_n) = \varepsilon_5(t_n)$, and otherwise let $\varepsilon_6(t)$ be defined by the line segments joining $(t_{n+1},\varepsilon_5(t_{n+1}) )$ and $(t_n,\varepsilon_5(t_n))$ for all $n$. Since $\varepsilon_4(t) \leq \varepsilon_6(t) \leq \varepsilon_5(t)$, it is easy to see that $\varepsilon_6(t)$ is strictly increasing, that $\varepsilon_6(t) \to 0$ as $t \to 0$, that for all $\alpha >0$, we have that $\frac{t^\alpha}{\varepsilon_6(t)} \to 0$ as $t \to 0$, that $\varepsilon_6 (t) \leq \frac{K}{10} \: \varepsilon_6 (t^K)$ for $t<t_0$, and thinking of the geometric interpretation of the slopes, it is also easy to see that $\frac{\varepsilon_6(t)}{t}$ is decreasing in $t$.

Finally, $\widetilde{\varepsilon}(t) = (\varepsilon_6)^{2-d} (t)$ satisfies the required conditions.

\end{proof}

Putting together Lemmata \ref{CanAlwaysReduceToTheoremGeneralizedHausdorffMeasuresCaseA} and \ref{CanAlwaysReduceToTheoremGeneralizedHausdorffMeasuresCaseB}, and Theorem \ref{SharpExampleTheoremForGeneralizedHausdorffMeasures}, we readily get

\begin{thm}\label{SharpnessOfTheorems1.1inACMOUand2.5inACMOUDifferentSense}
\begin{enumerate}
\item[(a)] Consider the statement \eqref{abscont}, i.e. that for any compact set $E \subset \C$ and any $K$-quasiconformal mapping $\phi:\C \to \C$, we have that $$\H^d(E)=0\,\,\,\Longrightarrow\,\,\,\H^{d'}(\phi(E))=0,$$
with $d'=\frac{2Kd}{2+(K-1)d}$ and $0<d<2$. If such a statement is true, then it is sharp in the sense that no strengthening of the conclusion (in terms of Hausdorff gauge functions) is possible with the same hypothesis. More precisely, for any gauge function $h(t)$ such that $\frac{t^{d'}}{h(t)} \to 0$ as $t \to 0$, there exists a compact set $E \subset \C$ and a $K$-quasiconformal mapping $\phi:\C \to \C$ such that $\H^d(E)=0$ but $\H^{h(t)}(\phi(E))>0$.

Notice that the statement \eqref{abscont} is true in the particular case $d=\frac{2}{K+1}$, $d'=1$, (Theorem \ref{theorem1.1inACMOU}), which is the relevant case for removability (see Theorems \ref{theorems1.2and4.3inACMOU} and \ref{theorem5.1inACMOU}), and it is conjectured to be true for all $0<d<2$ (Question 4.4 in \cite{astalaareadistortion} and Conjecture 2.3 in \cite{astalaclopmateuorobitguriartepreprint}.)

\item[(b)] Consider the statement that for any compact set $E \subset \C$ and any $K$-quasiconformal mapping $\phi:\C \to \C$, we have that 
\begin{equation}\label{StatementSigmaFiniteGoesToSigmaFiniteInAllDimensions}
\H^d(E) \ is \ \sigma-finite \,\,\,\Longrightarrow\,\,\,\H^{d'}(\phi(E)) \ is \ \sigma-finite, 
\end{equation}
with $d'=\frac{2Kd}{2+(K-1)d}$ and $0<d<2$. If such a statement is true, then it is sharp in the sense that no weakening of the hypothesis (in terms of Hausdorff gauge functions) is possible with the same conclusion. More precisely, for any gauge function $g(t)$ such that $\frac{g(t)}{t^d} \to 0$ as $t \to 0$, there exists a compact set $E \subset \C$ and a $K$-quasiconformal mapping $\phi:\C \to \C$ such that $\H^{g(t)}(E)=0$ but $\H^{d'}(\phi(E))$ is non-$\sigma$-finite.

Notice that the statement \eqref{StatementSigmaFiniteGoesToSigmaFiniteInAllDimensions} is true in the particular case $d=\frac{2}{K+1}$, $d'=1$, (Theorem \ref{theorem2.5inACMOU}), which is the relevant case for removability, (see Theorems \ref{theorems1.2and4.3inACMOU} and \ref{theorem5.1inACMOU}), and we conjecture it is true for all $0<d<2$.
\end{enumerate}

\end{thm}

\begin{proof}
For case $(a)$, given a gauge function $h(t)$ such that $\frac{t^{d'}}{h(t)} \to 0$ as $t \to 0$, we can assume, if necessary by ``getting closer" to $t^{d'}$, that for all $\alpha>0$, we have that $t^\alpha \: \frac{h(t)}{t^{d'}} \to 0$ as $t \to 0$. Define $\varepsilon (t)$ by the condition
\begin{equation}
\varepsilon^{ \frac{2}{2+(K-1)d} } (t^K) = \frac{h(t)}{t^{d'}} 
\end{equation}
it is then easy to see (thinking of compositions with functions of the type $t \to t^\beta$ for $\beta >0$) that $\varepsilon (t) \to \infty$ as $t \to 0$; that for all $\alpha>0$, $t^\alpha \varepsilon (t) \to 0$ as $t \to 0$; and that $t^d \varepsilon (t)$ is increasing in $t$ (since $h(t)$ is a gauge function.) Apply now Lemma \ref{CanAlwaysReduceToTheoremGeneralizedHausdorffMeasuresCaseA} and Theorem \ref{SharpExampleTheoremForGeneralizedHausdorffMeasures}.

For case $(b)$, given $g(t)$, we can assume, if necessary by getting ``closer" to $t^d$ that for all $\alpha>0$, we have that $\frac{t^\alpha}{ \delta_1 (t)} \to 0$ as $t \to 0$, where $\delta_1 (t) = \frac{g(t)}{t^d}$. As in the proof of Lemma \ref{CanAlwaysReduceToTheoremGeneralizedHausdorffMeasuresCaseB}, take $\delta_2 (t) \geq \delta_1 (t)$ so that $\delta_2 (t)$ is strictly increasing and $\delta_2 (t) \to 0$ as $t \to 0$. Now take $\delta_3 (t) = \sqrt{\delta_2 (t)}$, and apply to $t \to t^d \delta_3 (t)$ Lemma \ref{CanAlwaysReduceToTheoremGeneralizedHausdorffMeasuresCaseB} and Theorem \ref{SharpExampleTheoremForGeneralizedHausdorffMeasures}. The result follows by standard comparison theorems for Hausdorff measures (see right below \eqref{PartialOrderForGaugeFunctions2}.)

\end{proof}

\begin{remark}\label{NotPossibleToProveShaprnessByOneSingleExampleZeroGoesToZero}
One might wonder if it is possible to prove Theorem \ref{SharpnessOfTheorems1.1inACMOUand2.5inACMOUDifferentSense} $(a)$ ``in one shot", without having to resource to comparison theorems for Hausdorff measures (as in Theorem \ref{TheoremConstructionForUsualHausdorffMeasures} and Corollary \ref{Theorem1.1inACMOUSharpInSenseOfRelaxingHypothesis}). More precisely, whether a compact set $E$ would exist so that $\H^d(E)=0$ and $\H^{d'}(\phi E)=0$ for a certain $K$-quasiconformal mapping $\phi$, but so that $\H^{h}(\phi E)=\infty$ for all gauge functions $h(t)$ such that $\frac{t^{d'}}{h(t)} \to 0$ as $t \to 0$. This happens to be impossible due to a theorem of Besicovitch (\cite{rogershausdorffmeasures}, Theorem 42), which says that for any compact set $E$ such that $\H^f(E)=0$ for some gauge function $f$, there exists another gauge function $g$ with $g \prec f$ so that $\H^g(E)=0$.

\end{remark}

The sharp examples presented in this paper should help as test cases in terms of understanding the following Conjectures.

\begin{conjecture}\label{AstalaConjecture}\textbf{[Question 4.4. in \cite{astalaareadistortion}, Conjecture 2.3 in \cite{astalaclopmateuorobitguriartepreprint}]}
If $E \subset \C$ is a compact set, $\phi$ is a planar $K$-quasiconformal mapping, $0<d<2$ and $d'=\frac{2Kd}{2+(K-1)d}$, then 
\begin{equation}\label{abscont2}
\H^d(E)=0\,\,\,\Longrightarrow\,\,\,\H^{d'}(\phi(E))=0,
\end{equation}
or equivalently, $\phi^\ast\H^{d'}\ll\H^d$.
\end{conjecture}

Although perhaps somewhat optimistic, we think it is reasonable to conjecture that, given that the examples we presented are sharp for the cases in which Conjecture \ref{AstalaConjecture} is known to be true, then maybe they are also sharp for the other cases, i.e.

\begin{conjecture}\label{ConjectureSigmaFiniteGoesToSigmaFinite}
If $E \subset \C$ is a compact set, $\phi$ is a planar $K$-quasiconformal mapping, $0<d<2$ and $d'=\frac{2Kd}{2+(K-1)d}$ then
\begin{equation}\label{EquationConjectureSigmaFiniteGoesToSigmaFinite}
\H^d(E) \ is \ \sigma-finite \ \Longrightarrow \H^{d'}(\phi(E))  \ is \ \sigma-finite 
\end{equation}
\end{conjecture}

and also

\begin{conjecture}\label{ConjectureLlogL}
If $E \subset \C$ is a compact set, $\phi$ is a planar $K$-quasiconformal mapping, $0<d<2$ and $d'=\frac{2Kd}{2+(K-1)d}$ and 
$S(t)=t^d \log^\beta \left( \frac{1}{t} \right)$ for $\beta \neq 0$ (positive or negative), and $T(t) = t^{d'} \log^{ \frac{2\beta}{2+(K-1)d} } \left( \frac{1}{t} \right) $, then
\begin{equation}\label{EquationConjectureLlogL}
\H^{S}(E) =0 \,\,\,\Longrightarrow \,\,\, \H^{T}(\phi E) =0.
\end{equation}
\end{conjecture}

Note that Conjecture \ref{ConjectureLlogL} has been proven for $d=\frac{2}{K+1}$ and $\beta <0$ by Clop (\cite{clopthesis} p.69.) We thank him for this comment.

Of course one could formulate the corresponding conjectures for gauge functions of the type $h(t)= t^d \log^\beta \left( \frac{1}{t} \right) \left[ \log \log \left( \frac{1}{t} \right) \right]^\gamma $, etc. These ``$L^\alpha \left( \log L \right)^\beta$" gauge functions appear naturally in the context of mappings of finite distortion.

Note that in this paper we also recover Theorem 5.1 in \cite{astalaclopmateuorobitguriartepreprint}, even without the need for the auxiliary function $v(t)$ appearing in (5.10) in that paper. However, this comment is somewhat vacuous, since the construction in the present paper is based on the proof of Theorem 5.1 in \cite{astalaclopmateuorobitguriartepreprint}, but introduces a good number of modifications which makes it substantially more technical than Theorem 5.1 in \cite{astalaclopmateuorobitguriartepreprint}.

The construction presented here has applications to similar sharpness and removability problems in the contexts of H\"{o}lder continuous $K$-quasiconformal mappings (see \cite{clopremovablesingularitiesholderquasiregular} and \cite{clopnonremovablesetsholderquasiregular}) and mappings of finite distortion, among other contexts. We expect these applications to appear in a forthcoming paper.

\bibliographystyle{alpha}
\bibliography{references}

\vskip 1cm
\begin{itemize}


\item[]{Mathematics Department, 202 Mathematical Sciences Bldg., University of Missouri, Columbia, MO 65211-4100, USA 
{\it E-mail address:} {ignacio@math.missouri.edu}}
\end{itemize}

\end{document}